\documentclass[a4paper, 10pt]{article}

\usepackage{amsmath, amssymb, amsthm}
\usepackage[UKenglish]{babel}
\usepackage{vmargin}

\newcommand{\cR}{\mathcal{R}}

\newcommand{\F}{\mathbb{F}}
\newcommand{\K}{\mathbb{K}}
\newcommand{\M}{\mathbb{M}}
\newcommand{\Q}{\mathbb{Q}}
\newcommand{\Z}{\mathbb{Z}}

\newcommand{\G}{\mathbb{G}}

\newcommand{\Ad}{\operatorname{Ad}}

\DeclareMathOperator{\im}{im}
\DeclareMathOperator{\rk}{rk}

\DeclareMathOperator{\SL}{SL}
\DeclareMathOperator{\PSL}{PSL}

\DeclareMathOperator{\GL}{GL}

\DeclareMathOperator{\ad}{ad}
\DeclareMathOperator{\Aut}{Aut}
\DeclareMathOperator{\End}{End}
\DeclareMathOperator{\Id}{Id}
\DeclareMathOperator{\Nat}{Nat}
\DeclareMathOperator{\Sym}{Sym}

\newtheorem*{theorem*}{Theorem}
\newtheorem*{fact*}{Fact}
\newtheorem*{conjecture*}{Conjecture}

\newtheorem{theorem}{Theorem}[section]
\newtheorem{proposition}[theorem]{Proposition}
\newtheorem{lemma}[theorem]{Lemma}
\newtheorem{corollary}[theorem]{Corollary}
\newtheorem{definition}[theorem]{Definition}

\newtheorem*{observation*}{Observation}
\newtheorem*{definition*}{Definition}
\newtheorem*{lemma*}{Lemma}

\newtheorem{fact}{Fact}

\newtheorem*{notationinproof*}{Notation}

\newenvironment{proofclaim}{\begin{proof}}{\end{proof}}

\theoremstyle{definition}
\newtheorem*{remark*}{Remark}
\newtheorem*{remarks*}{Remarks}
\newtheorem*{notation*}{Notation}
\newtheorem{example}[theorem]{Example}
\newtheorem{observation}{Observation}
\newtheorem{notation}[theorem]{Notation}
\swapnumbers
\newtheorem{step}{--- Claim}[subsubsection]
\newtheorem{notationinproof}[step]{--- Notation}

\theoremstyle{remark}
\newtheorem*{localnotation*}{Notation}

\title{For a study of definable representations of algebraic groups}
\author{Adrien Deloro}

\begin{document}

\renewcommand{\o}{^{\circ}}
\renewcommand{\>}{\rangle}
\renewcommand{\L}{\mathbb{L}}
\renewcommand{\d}{\partial}

\maketitle
\renewcommand{\thefootnote}{}



\begin{abstract}
We classify irreducible $\SL_2(\K)$-modules of low Morley rank ($\leq 4 \rk \K$) as a first step towards a more general conjecture.
\end{abstract}
%

\section{Introduction}\label{S:Introduction}

It was more than fifty years since Robert Steinberg published his celebrated ``tensor product'' theorem.

\begin{fact*}[Steinberg's Tensor Product Theorem, {\cite{SRepresentations}}]
Let $G$ be a semisimple algebraic group of characteristic $p \neq 0$ and rank $\ell$, and let $\cR$ denote the set of $p^\ell$ irreducible rational projective representations of $G$ in which the high weight $\lambda$ satisfies $0 \leq \lambda(a) \leq (p-1)$ for all simple roots $a$. Let $\alpha_i$ denote the automorphism $t\mapsto t^{p^i}$ of the universal field as well as the corresponding automorphism of $G$, and for $R \in \cR$ let $R^{\alpha_i}$ denote the composition of $\alpha_i$ and $R$. Then every irreducible rational projective representation of $G$ can be written uniquely as $\bigotimes R_i^{\alpha_i}$. Conversely, every such product yields an irreducible rational projective representation of $G$.
\end{fact*}

As conjectured by Steinberg, the rationality assumption was later removed by Borel and Tits following their work on ``abstract homomorphisms'' between groups of $\K$-points \cite{BTHomomorphismes}. This of course requires consideration of all field automorphisms and not only iterates of the Frobenius map in the above theorem, but the set $\cR$ remains unchanged \cite[Corollaire 10.4]{BTHomomorphismes}.

We hope the reader to be familiar with the significance of Steinberg's tensor theorem. For we wish in the present paper to conjecture an analogue in a different setting; such an analogue would be the first stone of a bridge linking representation theory to model theory.
As a matter of fact, we suggest to study the category of \emph{definable} representations of the semisimple algebraic groups --- and since our expected reader is not an expert in mathematical logic, (s)he will find all necessary explanations after the statements below.

\begin{conjecture*}
Let $(G, V)$ be a module of finite Morley rank, where $G$ is isomorphic to the group $\G_\K$ of $\K$-points of some semisimple algebraic group $\G$. Suppose that the action is faithful and irreducible.

Then there are a definable, compatible $\K$-vector space structure on $V$ and definable field automorphisms $\varphi_i$ such that $V \simeq \otimes {}^{\varphi_i} R_i$, definably, as a $\K[G]$-module.
\end{conjecture*}

We prove a basic case of this conjecture (the precise statement is in \S\ref{s:theorem}).

\begin{theorem*}
Work in the setting of the conjecture. Suppose that $\K$ does not have characteristic $2$ or $3$, that $G \simeq \SL_2(\K)$ and that $\rk V \leq 4 \rk \K$, where $\rk$ denotes the Morley rank.
Then the desired conclusion holds.
\end{theorem*}

Partial results had already been obtained: identification of the natural action when $\rk V \leq 2\rk \K$ in \cite{DActions} (with an application to identification of the standard action of $\GL_n(\K)$ \cite{BBGroups}), and identification of the adjoint action when $\rk V \leq 3\rk \K$ in \cite{CDSmall}. One easily imagines that $3 \rk \K = \rk \SL_2(\K)$ is a critical value past which things become much more difficult. Not only does our proof begin to overcome such difficulties, it also provides much quicker arguments for \cite{DActions, CDSmall} (see \S\ref{s:methods}).

And now let us explain what we mean.

\subsection{Model Theory: Ranked Universes}

Model theory 
is concerned with mathematical structures in both a wide (general) and narrow (the restriction to ``first-order'') sense. But we do not need a general introduction as we shall focus on groups. An excellent reference describing our setting is \cite{BNGroups}.

For the model theorist, algebraic groups as functors hardly exist: only affine groups of $\K$-points can be examined, one by one. Such objects are equipped with a collection of finitary relations --- among which the equality and the group law of course, but there could be more. This we call a ``group structure''.

Now on a general group structure $G$ model theory studies the class of so-called \emph{definable} sets \cite[\S4.1.1]{BNGroups}: singletons, Cartesian powers of $G$, and the graphs of relations are definable, the class is closed under boolean combinations, under projections, and under quotients by definable equivalence relations.
This reflects the logician's will to work only with ``first-order'' formulas (an approach we shall not need here). This also loosely mimicks the behaviour of the constructible class in affine algebraic groups. Of course the geometry of the constructible class in an affine algebraic group over an algebraically closed field is well-understood, whereas for a general, abstract group structure, nothing similar should be expected.

One ingredient is obviously missing to bring our setting closer to affine algebraic group theory: the Zariski topology. But this is not in the scope of model theory, and we shall rely on no topological methods. The loss is serious as geometers imagine. Instead we merely use a coarse analogue of the Zariski dimension, known as the \emph{Morley rank}.

The Morley rank, in general an ordinal-valued function, is defined by a standard induction:
\begin{itemize}
\item
a non-empty, finite definable set has rank $0$;
\item
a definable set $A$ has rank $\geq \alpha +1$ iff there are infinitely many disjoint definable subsets $B_i \subseteq A$ with $\rk B_i \geq \alpha$;
\item
to deal with limit ordinals take suprema.
\end{itemize}

Now it so happens (at least in the case of a group structure; other mathematical objects are not all so well-behaved model-theoretically) that if $G$ is a group structure the Morley rank of which is finite, then it satisfies the following properties \cite[\S4.1.2]{BNGroups}, where $f: A\twoheadrightarrow B$ is any definable surjection between definable sets:
\begin{itemize}
\item
the rank is definable: for any $k$, the set $\{b \in B: \rk f^{-1}(b) = k\}$ is definable;
\item
the rank is additive: if the latter set equals $A$ then $\rk A = \rk B + k$;
\item
uniformity phenomena (called ``elimination of infinite quantifiers'' by logicians): there is an integer $m$ such that for any $b \in B$, the fiber $f^{-1}(b)$ either is infinite or of cardinality $\leq m$.
\end{itemize}

All standard dimension computations \cite[\S4.2]{BNGroups} are therefore allowed, though we fall short of topological methods.
As finiteness of the Morley rank is inherited by definable subsets, the phrases ``let $G$ be a group of finite Morley rank'' and ``let $G$ be a group definable in a finite Morley rank universe'' describe the same setting.

\subsection{Ranked Groups}

A ranked group structure enjoys a number of properties highly reminiscent of algebraic group theory:
\begin{itemize}
\item
the existence of a connected component \cite[\S5.2]{BNGroups} $H\o$ for each definable subgroup $H \leq G$;
\item
generation by so-called \emph{indecomposable} definable sets, similar to generation by irreducible subsets in algebraic groups \cite[Theorem 5.26, known as Zilber's Indecomposibility Theorem]{BNGroups} (indecomposibility is a clever way to imitate topological irreducibility without a topology: a definable set is \emph{indecomposable} if no definable subgroup can partition it using a \emph{finite} number $k > 1$ of translates);
\item
the ability to retrieve fields from definable, connected, infinite, soluble, non-nilpotent groups, as in the following fundamental fact.
\end{itemize}

\begin{fact}[{Zilber's Field Theorem: \cite[Theorem 9.1]{BNGroups}}]
Let $G = A\rtimes H$ be a group of finite Morley rank where $A$ and $H$ are infinite definable abelian subgroups. Suppose the action is faithful and $H$-irreducible, viz. $A$ has no definable, proper, infinite, $H$-invariant subgroup.
Then the ring of endomorphisms of $A$ generated by $H$ is a definable field $\K$, for which $A \simeq \K_+$ definably and $H \hookrightarrow \K^\times$ in its action on $\K_+$.
\end{fact}

Now an infinite field structure of finite Morley rank is necessarily algebraically closed by a theorem of Macintyre \cite[Theorem 8.1]{BNGroups} (note for the geometer: $\K$ may have Morley rank $>1$). In particular there are no extensions of infinite definable fields in our setting, and consequently infinite fields of finite Morley rank have no definable groups of automorphisms \cite[Theorem 8.3]{BNGroups} (note: definable field automorphisms however exist, but only in characteristic $p$ \cite[p.~129]{BNGroups}: it is not known whether they are all relative powers of the Frobenius map, which will make our conjecture interesting).
But model theorists have been extremely plafyful with subgroups of $\K_+$, resp. $\K^\times$, which are not to be thought as ``minimal'' with respect to definability. Some of our arguments will look quite convoluted to the geometer, precisely because of this.

\emph{Quite remarkably, we need virtually no other reprequisites than the above material} --- which we take as a strong indication that we are at the beginning of something not only new, but also of interest.

\subsection{Ranked Modules}

A ranked module $(G, V, \cdot)$ (or simply $(G, V)$ when there is no ambiguity) is a semi-direct product $G \ltimes V$ such that:
\begin{itemize}
\item
$G \ltimes V$ has finite Morley rank;
\item
$G$ and $V$ are definable and infinite;
\item
and $V$ is abelian and connected.
\end{itemize}

Be very careful that it is a requirement on $V$ to be both definable and connected. All related notions are therefore connected versions: an irreducible module, a composition series, and so on. But in practice working in the definable-connected category is completely intuitive.

The classification of ranked modules could be next to impossible.
As a matter of fact, despite highly regarded efforts towards a well-known classification conjecture for the infinite simple groups of finite Morley rank, such abstract groups are not determined. But in the present paper we handle a very special case, where $G$ is known to be isomorphic with $\SL_2(\K)$. In particular, our work is almost disjoint in spirit from the ambitious classification programme we just mentioned, and definitely not about the so-called Cherlin-Zilber algebraicity conjecture.

Now if $\SL_2(\K)$ is definable inside a group of finite Morley rank, then so is $\K$ (the reader not familiar with definability should exercise here as it involves all techniques we shall use in the article); by the previous subsection, it is algebraically closed.
Hence, if $(G, V)$ is a (faithful and irreducible) ranked module with $G$ of the form $\SL_2(\K)$, or more generally $\G_\K$ for $\G$ a semisimple algebraic group, then $\K$ is an algebraically closed field. In such an ideal world, how can $V$ fail to be a representation of the algebraic group $G$?

\subsection{On the Conjecture}\label{s:conjecture}

Let us repeat our conjecture.

\begin{conjecture*}
Let $(G, V)$ be a module of finite Morley rank, where $G$ is isomorphic to the group $\G_\K$ of $\K$-points of some semisimple algebraic group $\G$. Suppose that the action is faithful and irreducible.

Then there are a definable, compatible $\K$-vector space structure on $V$ and definable field automorphisms $\varphi_i$ such that $V \simeq \otimes {}^{\varphi_i} R_i$, definably, as a $\K[G]$-module.
\end{conjecture*}

\paragraph{1. Null Characteristic.}

Like Steinberg's original theorem, the conjecture degenerates in characteristic $0$, a case where it is already known.

\begin{fact}[{\cite[Lemma 1.4]{CDSmall}}]\label{f:char0}
In a universe of finite Morley rank, consider the following definable objects: a field $\K$, a quasi-simple algebraic group over $\K$, a torsion-free abelian group $V$, and a faithful action of $G$ on $V$ for which $V$ is $G$-irreducible. Then $V \rtimes G$ is algebraic with respect to some $\K$-vector space structure on $V$.
\end{fact}

One should note that this is a weak analogue of the equivalent statement in the so-called $o$-minimal category, where prior knowledge of $G$ is not even required.

\begin{fact}[{not used; Macpherson-Mosley-Tent; \cite[Proposition 4.1]{MMTPermutation}}]
Let $(G, V)$ be a module definable in an $o$-minimal structure. Suppose that the action is faithful and irreducible.
Then there is a definable real closed field $R$ such that $V$ has definably the structure of a vector space over $R$, and $G$ is definably isomorphic to a subgroup of $\GL(V)$ in the standard action on $V$. Furthermore, if $G$ is connected then its image in $\GL_m(R)$ is $R$-semialgebraic.
\end{fact}


\paragraph{2. Linearity.}

Let us return to the ranked setting. In positive characteristic, little is known.
Needless to say, Steinberg's slightly computational proof, and the more modern approach through Frobenius kernels as schemes \cite{CPSTensor} are both completely out of our methodological scope. And we confess to have no general strategy for the moment.

Whether definable $\K$-linearity of $V$ would suffice is not completely clear: by the Borel-Tits corollary to Steinberg's theorem, \emph{all} irreducible linear representations of $\G_\K$ are known and of the desired form --- but the definability of the field automorphisms used as twists is not fully obvious albeit highly expectable.
We do however suspect that for given $G = \G_\K$ and $V \simeq \otimes_i {}^{\varphi_i} R_i$ as above, the structures $(\K; \frac{\varphi_2}{\varphi_1}, \dots, \frac{\varphi_n}{\varphi_1})$ and $(G, V, \cdot)$ should be inter-definable. It is a topic we shall return to some day.

But in any case, we see no general path to linearity, a much harder question. Our proof proceeds by defining the $\K$-vector space structure manually, that is, with some hint of what the twist will be.

\paragraph{3. Field automorphisms.}

We wish to make a few more general comments.

\begin{fact}[{not used; Macpherson-Pillay, Poizat; \cite[Th\'eor\`eme 2]{PQuelques}}]\label{f:Poizat}
Definable in a universe of finite Morley rank, let $\K$ be a field of prime characteristic and $G \leq \GL_n(\K)$ be a simple subgroup. Then there are definable field automorphisms $\varphi_1, \dots, \varphi_k$ such that $G$ is definable in the structure $(\K; +, \cdot, \varphi_1, \dots, \varphi_k)$.
\end{fact}

However remarkable this phenomenon and the proof Poizat gave, one hardly sees how to use it in our arguments. But we take it as supporting evidence for the conjecture: given a definable representation, the relevant category seems to be the rational category augmented by a finite number of twists.

We wish to go further into mere speculation. In Steinberg's classical result one computes the $p$-adic expansion of the highest weight; what would be the similar computation here?
Observe how any two definable field automorphisms of a field of finite Morley rank commute over the algebraic closure of the prime field, hence everywhere; in particular, the ring $\F_p[\Aut_{\rm def}(\K)]$ generated by definable automorphisms of $\K$ \emph{is commutative}. So computing in this ring would make (some) sense.
This also do we take as a strong indication that our conjecture is the right statement.

And now let us be specific.

\subsection{The Result and its Legitimacy}\label{s:theorem}

\begin{theorem*}
Definable in a universe of finite Morley rank, let $\K$ be a field, $G = \SL_2(\K)$, and $V$ be an irreducible $G$-module of rank $\leq 4\rk \K$.
Suppose the characteristic of $\K$ is neither $2$ nor $3$.

Then there is a definable $\K$-vector space structure on $V$ such that $V$, as a $\K[G]$-module, $V$ is definably isomorphic to one of the following:
\begin{itemize}
\item
$\Nat G$, the natural representation;
\item
$\Ad G \simeq \Sym^2 \Nat G$, the adjoint representation;
\item
$\Sym^3 \Nat G$, the rational representation of degree $4$ on $\K[X^3, X^2Y, XY^2, Y^3]$;
\item
$\Nat G \otimes{}^\chi (\Nat G)$, a twist-and-tensor representation for some definable $\chi \in \Aut(\K)$.
\end{itemize}
\end{theorem*}

We do not know what happens in characteristics $2$ and $3$. But we suspect things would be much harder, and despite the challenge we confess that for the time being our interests lie more in higher Morley rank than in lower characteristic.

The theorem is proved in \S\ref{S:Proof}, but it requires the tools developed in \S\ref{S:Tools}.

We must legitimate our result --- which we agree is but a very first step. However we consider it a breakthrough; there are three critical steps in the analysis:
\begin{enumerate}
\item\label{i}
leave the category of rational representations of $\SL_2(\K)$;
\item\label{ii}
remove bounds on the Morley rank of the $\SL_2(\K)$-module $V$;
\item
handle the other semisimple algebraic groups.
\end{enumerate}

As we said, until now only the case $G \simeq \SL_2(\K)$ and $\rk V \leq 3\rk \K$ was known \cite{CDSmall}. \emph{Here we successfully leave the rational category}, which requires a complete change of methods.
We view completion of \ref{i} as evidence for the maturity of (some of) our methods. There are no claims now on how to pass the remaining steps.

\subsection{On our Methods}\label{s:methods}

Let us describe the main features of the argument.

\begin{itemize}
\item
As opposed to \cite{CDSmall} which handled the $\rk V \leq 3\rk \K$ case (see in particular \cite[Step 5 in \S2 and the Proposition in \S3]{CDSmall}), we are now working with modules greater in rank than the group. Consequently there can be no useful genericity argument since orbits of elements are all small when compared to the module.

The present study entirely substitutes genericity arguments with much better techniques given in \S\ref{S:Tools}. Some of these were found by Tindzogho Ntsiri in his PhD \cite{TEtude}, also \cite{TStructure}; what we cannot extend we faithfully reproduce for the sake of future reference.
\item
We convert the abstract tools into methods based on the structure of $V$ as a $T$-module (where $T$ is an algebraic torus) and explained in \S\ref{s:applied}: we shall even give a considerably shorter proof of \cite{CDSmall} in \S\ref{s:3k}.

It is not unreasonable to hope that these methods can be generalised further.
\item
The main issue with our proof is that we reproduce some of the unpleasant computations developed for $\SL_2(\Z)$-modules of short nilpotence length \cite{TV-III}.
Now one should take intrinsic limitations on the unipotence length with care: \cite[Proposition 2]{TV-III}, which we briefly touch upon in \S\ref{s:notations}, explains why computations which look possible in unipotence length $\leq 5$ have a geometric obstruction in longer length.

Although this geometric phenomenon simply cannot survive strong model-theoretic assumptions, for the moment there is no strategy other than the ``short length miracle'' of \cite[Theorem 1]{TV-III}.
In our opinion this is the main obstacle to jumping to identification of \emph{all} rational representations of $\SL_2(\K)$ among its definable modules.

But perhaps model theory, or a combination of finite group theory and model theory, could substitute such computations. This would open the way to something general and step \ref{ii} in \S\ref{s:theorem} might then prove tractable.
\item
In the twisted case we untensor by hand a linear structure we defined by hand. As we said in the previous subsection, one could naively hope to divide a solution to any case of the conjecture into two tasks:
\begin{itemize}
\item
proving linearity (in as general a context as possible);
\item
proving in general that if $V \simeq \otimes {}^{\varphi_i} R_i$ is a definable $\G_\K$-module, then the $\varphi_i/\varphi_1$ are definable from $G\ltimes V$ (see \S\ref{s:conjecture}.2).
\end{itemize}

But as we strongly suspect there is no direct path to linearity, analysing the conjecture in this two-step fashion hardly helps.

This certainly reflects on our methods for the tensor case (Proposition \ref{p:twist} of \S\ref{s:twist}): we isolate weight spaces and define the field action following what the twist should be; only after doing so do we actually retrieve the field automorphism.
\end{itemize}

\section{General Tools}\label{S:Tools}

The proof of our main Theorem will be in \S\ref{S:Proof}; this preliminary section contains tools whose interest goes beyond the rank $4k$ case. In \S\ref{s:Maschke} we state Tindzogho Ntsiri's semi-simple analogue of Maschke's Theorem. In \S\ref{s:covariant} we enhance and simplify his systematic study of covariant, bi-additive maps between definable fields. These general results are applied to $\SL_2(\K)$-modules of finite Morley rank in \S\ref{s:applied}.

Before we start, we remind the reader that by convention, all modules are definable and connected; this affects the notions of irreducibility and composition series.

\subsection{Abstract Tools 1: Actions of Tori}\label{s:Maschke}

\begin{fact}[{\cite[Lemma G (iv)]{DJInvolutive}}]\label{f:semisimple}
In a universe of finite Morley rank, let $A$ be a definable, abelian group and $R$ be a group acting on $A$ by definable automorphisms. Let $A_0 \leq A$ be a definable, $R$-invariant subgroup.

Suppose that $A$ has exponent $p$ and $R$ is a definable, soluble, $p^\perp$ group. Then $A = [A, R] \oplus C_A(R)$, $([A, R]\cap A_0) = [A_0, R]$, $C_A(R)$ covers $C_{A/A_0}(R)$, and $C_R(A) = C_R(A_0, A/A_0)$.
\end{fact}

It certainly is superfluous to explain the relevance of a definable version of Maschke's Theorem.
The result below is from Tindzogho Ntsiri's PhD thesis \cite{TEtude}.

\begin{fact}[{Tindzogho Ntsiri, \cite[\S5.2.1, Proposition 5.23]{TEtude} or \cite[Corollary 2.17]{TStructure}}]\label{f:Maschke}
In a universe of finite Morley rank, let $T$ be a definable, connected, soluble, $p^\perp$ group and $V$ be a $T$-module of exponent $p$.
Suppose $C_V(T) = 0$.
Then every $T$-submodule admits a direct complement $T$-module.
\end{fact}

\begin{remark*}
The author wonders whether anything similar might exist in characteristic $0$. This is mostly for art's sake as the conjecture is settled there by Fact \ref{f:char0}.
\end{remark*}

\subsection{Abstract Tools 2: Covariant Bi-additive Maps}\label{s:covariant}

We move to studying $T$-covariant, bi-additive maps; the reason for doing so is the following. When analysing an arbitrary $\SL_2(\K)$-module, one will try to capture weight spaces as $T$-minimal modules; if the action of $T$ is not trivial, Zilber's Field Theorem recovers an algebraically closed field from every such module. But the interpretable field has no reason to be isomorphic to $\K$, nor even to be the same on the various $T$-minimal modules.
The hope is that linking the various weight spaces via the commutation map with a unipotent subgroup normalised by $T$ should force interpretable fields to be definably isomorphic with $\K$.

A first result in this direction was developed (entirely by Cherlin) for \cite{CDSmall} and served as a basis of Tindzogho Ntsiri's work.

In a sense, the material in this section is again from \cite{TEtude}; but as opposed to Fact \ref{f:Maschke}, here we can generalise Tindzogho Ntsiri's conclusions while simplifying his proof. The resulting statement is as follows.


\begin{theorem}[{cf. \cite[\S5.1.1, Th\'eor\`eme 5.7]{TEtude} and \cite[Theorem 3.2]{TStructure}}]\label{t:threefields}
All objects (including functions and actions) are supposed to be definable in a universe of finite Morley rank.

Let $\K, \L, \M$ be fields and $U, V, W$ be vector spaces over $\K$, $\L$, $\M$, respectively.
Suppose there is a group $T$ acting (not necessarily faithfully) by scalar multiplication on $U$, $V$, $W$, i.e. as a subgroup of $\K^\times$, $\L^\times$, $\M^\times$. Let $\kappa: T\to \K^\times$, $\lambda: T \to \L^\times$ and $\mu: T\to \M^\times$ be the associated group homomorphisms, and let $\rho = \kappa \times \lambda \times \mu : T \to \K^\times \times \L^\times \times \M^\times$.

Let $\beta : U \times V \to W$ be a non-zero bi-additive map and suppose that $\beta$ is \emph{plus-or-minus-$T$-covariant}, i.e.
%
there are $\varepsilon_\K, \varepsilon_\L, \varepsilon_\M  \in \{\pm 1\}$ such that for all $(t, u, v) \in T\times U \times V$, one has $\beta(t^{\varepsilon_\K} \cdot u, t^{\varepsilon_\L}\cdot v) = t^{\varepsilon_\M} \cdot \beta(u, v)$.

Then:
\begin{itemize}
\item
$\ker \rho = \ker \kappa \cap \ker \lambda = \ker \kappa \cap \ker \mu$;
\item
if $\ker \kappa / \ker \rho$ is infinite, then $\L \simeq \M$;
\item
if both $\ker \kappa / \ker \rho$ and $\im \kappa$ are infinite, then $\K \simeq \L \simeq \M$;
\item
if $\ker \mu / \ker \rho$ is infinite, then $\K \simeq \L$;
\item
if both $\ker \mu / \ker \rho$ and $\im \mu$ are infinite, then $\K \simeq \L \simeq \M$.
\end{itemize}
\end{theorem}

The proof will take place in \S\ref{s:threefields:proof}; we start with easy remarks.

\subsubsection{Two Remarks on Fields}

\begin{notation*}
For a field $\F$, let $\F' = \F\setminus\{0\} \cup \{\infty\}$ and define $\ast:\F'\times \F'\to \F'$ by:
\[\left\{\begin{array}{ll}
a \ast b = \frac{ab}{a+b} & \mbox{ for } a, b \in \F\setminus \{0\} \mbox{ with } a \neq -b\\
a \ast (-a) = \infty & \mbox{ for } a \in \F\setminus\{0\}\\
a \ast \infty = \infty \ast a = a & \mbox{ for } a \in \F'\\
\end{array}\right.\]
Extend $\cdot$ to $\F'$ by $a \cdot \infty = \infty \cdot a = \infty$.
\end{notation*}

\begin{observation}
$(\F', \ast, \cdot)$ is a field isomorphic to $(\F, +, \cdot)$ via the function $\iota_\F (a) = a^{-1}$ (with the convention $0^{-1} = \infty$ and conversely); if $\F$ is definable then so are $\F'$ and this isomorphism.
\end{observation}


\begin{observation}
Let $\F_1$ and $\F_2$ be fields, and denote by $\pi_i : \F_1 \times \F_2 \to \F_i$ the projection maps. If a subring $R \leq \F_1\times \F_2$ is a field, then $\pi_1(R) \simeq \pi_2(R)$ (actually $R$ stands for the graph of one such field isomorphism; alternatively, $\pi_1(R) \simeq R \simeq \pi_2(R)$). In particular $\F_1$ and $\F_2$ have the same characteristic.

If in addition $\F_1$, $\F_2$, and $R$ are definable in a theory of finite Morley rank with $R$ infinite, then $\F_1 \simeq \F_2$.
\end{observation}
\begin{proofclaim}
It suffices to see that $R$ is a one-to-one, functional relation. Since $R$ is an additive subgroup of $\F_1\times \F_2$, it suffices to see that $(0, y) \in R$ entails $y = 0$: which is obvious since $R$ is a field.

For the second claim, bear in mind that there can be no extensions of infinite, definable fields for all such are algebraically closed \cite[Theorem 8.1]{BNGroups}.
\end{proofclaim}

\subsubsection{An Example to Understand}

The following will not be used anywhere in the present article, but it involves all ideas underlying the proof of Theorem \ref{t:threefields}.

\begin{lemma*}\label{l:covariantadditive}
All objects (including functions and actions) are supposed to be definable in a universe of finite Morley rank.

Let $\K, \L$ be fields and $U, V$ be vector spaces over $\K, \L$, respectively. Suppose there is a group $T$ acting (not necessarily faithfully) by scalar multiplication on $U, V$, i.e. as a subgroup of $\K^\times$, $\L^\times$. Let $\kappa: T \to \K^\times$ and $\lambda : T\to \L^\times$ be the associated group homomorphisms, and let $\rho = \kappa \times \lambda : T \to \K^\times \times \L^\times$.

Let $\alpha: U \to V$ be a non-zero additive map and suppose that $\alpha$ is either $T$-covariant, i.e. $\alpha(t\cdot u) = t\cdot \alpha(u)$, or $T$-contravariant, i.e. $\alpha(t\cdot u) = t^{-1}\cdot \alpha(u)$.

Then:
\begin{itemize}
\item
$\ker \rho = \ker \kappa = \ker \lambda$;
\item
if $\im \kappa$ (or equivalently, $\im \lambda$) is infinite, then $\K\simeq \L$ definably.
\end{itemize}
\end{lemma*}
\begin{proof}
Since $\alpha$ is non-zero, there is $u_0 \in U$ with $\alpha(u_0) \neq 0$.
Let $t \in \ker \kappa$. Then $\lambda(t) \cdot \alpha(u_0) = \alpha(\kappa(t)^{\pm 1}\cdot u_0) = \alpha(u_0) \neq 0$, where the value $\pm 1$ depends on whether $\alpha$ is $T$-covariant or $T$-contravariant; hence $\lambda(t) = 1$. Conversely, if $t \in \ker \lambda \setminus \ker \kappa$, then letting $u_1 = (\kappa(t) - 1)^{-1}\cdot u_0$ one finds:
\[\alpha(u_0) = \alpha((\kappa(t) - 1)\cdot u_1) = \lambda(t)^{\pm 1}\cdot \alpha(u_1) - \alpha(u_1) = 0\]
which is a contradiction.

We prove the second claim for a covariant map. Let $\F_{\K, \L} = \{(k, \ell) \in \K\times \L: \forall u \in U, \alpha(k \cdot u) = \ell \cdot \alpha(u)\}$. One easily checks that $\F_{\K, \L}$ is a definable subfield of the ring $\K\times \L$.
Also let $\G_{\K, \L} = \F_{\K, \L}^\times$.
If $\im \kappa \simeq T/\ker \kappa \simeq \im \lambda$ is infinite, then so is $\rho(T) \leq \G_{\K, \L}$, and therefore $\F_{\K, \L}$ too. It follows $\K \simeq \L$.

The proof of the contravariant case is even more instructive. Let $\F_{\K, \L'} = \{(k, \ell) \in \K\times \L\setminus\{0\}: \forall u \in U, \ell\cdot \alpha(k\cdot u) = \alpha(u)\}\cup\{(0, \infty)\}$. We claim that $\F_{\K, \L'}$ is a definable subfield of $\K\times \L'$; perhaps the only non-trivial claim is stability under the additive law. So suppose $(k_i, \ell_i) \in \K\times \L'$ for $i\in\{1, 2\}$; the cases where $\ell_i = \infty$ or $\ell_1 + \ell_2 = 0$ are easily dealt with, so we assume not. Then for $u \in U$:
\begin{align*}
(\ell_1\ast \ell_2) \cdot \alpha((k_1 + k_2)\cdot u) & = \frac{\ell_1 \ell_2}{\ell_1 + \ell_2} \cdot \left(\alpha(k_1\cdot u) + \alpha(k_2 \cdot u)\right)\\
& = \frac{\ell_2}{\ell_1 + \ell_2} \cdot \alpha(u) + \frac{\ell_1}{\ell_1 + \ell_2} \cdot \alpha(u)\\
& = \alpha(u)
\end{align*}
as desired.
Also let $\G_{\K, \L'} = \{(k, \ell) \in \K^\times \times \L'^\times: \forall u \in U, \ell \cdot \alpha(k\cdot u) = \alpha(u)\}$; clearly $\G_{\K, \L'} = \F_{\K, \L'}^\times$. If $\im \kappa$ is infinite, then by contravariance so are $\rho(T) \leq \G_{\K, \L'}$ and $\F_{\K, \L'}$. Hence $\K \simeq \L' \simeq \L$.

Using the same argument twice is unsatisfactory. So suppose again that $\alpha$ is $T$-contravariant. By means of the inverse isomorphism $\iota_\L: \L\simeq \L'$, view $V$ as a vector space over $\L'$. Then $\alpha: U \to V$ is $T$-covariant. Since the other assumptions and the statements are invariant under this change, the contravariant case actually reduces to the covariant one (first claim included).
\end{proof}

\subsubsection{Proof of Theorem \ref{t:threefields}}\label{s:threefields:proof}


\begin{proof}[Proof of Theorem \ref{t:threefields}]
%
%
Throughout the proof we assume $\varepsilon_\K = \varepsilon_\L = \varepsilon_\M = 1$; we explain why in the end.

The proof that $\ker \rho = \ker \kappa \cap \ker \lambda = \ker \kappa \cap \ker \mu$ is essentially as in the previous Lemma. Since $\beta$ is non-zero there is $(u_0, v_0) \in U\times V$ with $w_0 = \beta(u_0, v_0) \neq 0$.
If $t \in \ker \kappa \cap \ker \lambda$, then $\mu(t)\cdot w_0 = \beta(\kappa(t)\cdot u_0, \lambda(t) \cdot v_0) = w_0 \neq 0$, so $\mu(t) = 1$. Conversely for $t \in (\ker \kappa \cap \ker \mu) \setminus \ker \lambda$, let $v_1 = (\lambda(t) - 1)^{-1}\cdot v_0$. Then:
\begin{align*}
\beta(u_0, v_0) & = \beta(u_0, (\lambda(t) - 1)\cdot v_1) = \beta(\kappa(t) \cdot u_0, \lambda(t)\cdot v_1) - \beta(u_0, v_1)\\
& = \mu(t) \cdot \beta(u_0, v_1) - \beta(u_0, v_1) = 0
\end{align*}
which is a contradiction.
We now need a theory of notations.

\begin{localnotation*}\
\begin{itemize}
\item
Let $\F_{\K, \M} = \{(k, m) \in \K\times \M: \forall (u, v) \in U\times V, \beta(k\cdot u, v) = m \cdot \beta(u, v)\}$.
\item
Let $\F_{\K, \L'} = \{(k, \ell) \in \K\times (\L\setminus\{0\}): \forall (u, v) \in U\times V, \beta(k\cdot u, \ell \cdot v) = \beta(u,v)\}\cup \{(0, \infty)\}$.
%
%
\item
Let $\G_{\K, \L, \M} = \{(k, \ell, m) \in \K^\times \times \L^\times \times \M^\times: \forall (u, v) \in U\times V, \beta(k\cdot u, \ell \cdot v) = m \cdot \beta(u, v)\}$.
\item
Let $\pi_{\K^\times}: \K^\times \times \L^\times \times \M^\times \to \K^\times$ denote projection on $\K^\times$ (define $\pi_{\L^\times}$ and $\pi_{\M^\times}$ likewise).
\item
Let $\G_{\K, \M} = \G_{\K, \L, \M} \cap \ker \pi_{\L^\times} = \{(k, 1, m): (k, m) \in \F_{\K, \M}^\times\}$.
\item
Let $\G_{\K, \L'} = \G_{\K, \L, \M} \cap \ker \pi_{\M^\times} = \{(k, \ell, 1): (k, \ell) \in \F_{\K, \L'}^\times\}$.
\end{itemize}
\end{localnotation*}

One should also define $\F_{\L, \M}$ and $\G_{\L, \M}$ similarly.
As in the above Lemma, $\F_{\K, \M}$ is a definable subfield of $\K\times \M$ and $\F_{\K, \L'}$ is one of $\K\times \L'$.
Clearly $\G_{\K, \L, \M}$ is a subgroup of $\K^\times \times \L^\times \times \M^\times$; by the covariance assumption, $\rho(T) \leq \G_{\K, \L, \M}$.
We are finally ready for the proof.

Suppose $\ker \kappa/\ker \rho$ to be infinite. This means that $\rho(\ker \kappa) \leq \G_{\K, \L, \M} \cap \ker \pi_{\K^\times} = \G_{\L, \M}$ is infinite, and therefore so is $\F_{\L, \M}$: it follows that $\L \simeq \M$.

Suppose that both $\ker \kappa/\ker \rho$ and $\im \kappa$ are infinite.
Here again $\F_{\L, \M}$ is infinite, so it surjects onto $\L$: hence $\pi_{\L}(\F_{\L, \M}) = \L$ and $\pi_{\L^\times}(\G_{\L, \M}) = \L^\times$. But since $\im \kappa$ is infinite too, $\G_{\L, \M}$ has infinite index in $\rho(T) \cdot \G_{\L, \M} \leq \G_{\K, \L, \M}$. Therefore $\ker \pi_{\L^\times} \cap \G_{\K, \L, \M} = \G_{\K, \M}$ is infinite; so is $\F_{\K, \M}$, proving $\K \simeq \M$.

Now suppose that $\ker \mu /\ker \rho$ is infinite. This means that $\rho(\ker \mu) \leq \G_{\K, \L'}$ is infinite, and therefore so is $\F_{\K, \L'}$: it follows that $\K \simeq \L' \simeq \L$.

Finally suppose that both $\ker \mu /\ker \rho$ and $\im \mu$ are infinite.
Here again $\pi_{\L^\times}(\F_{\K,\L'}) = \L^\times$ (one should remember that $\L^\times$ and $\L'^\times$ have the same underlying set). But in addition, $\im \mu$ is infinite, so $\G_{\K, \L'}$ has infinite index in $\rho(T) \cdot \G_{\K, \L'} \leq \G_{\K, \L, \M}$: proving that $\ker \pi_{\L^\times} \cap \G_{\K, \L, \M} = \G_{\K, \M}$ is infinite, and $\K \simeq \M$.

It remains to cover the other cases. Whenever $\varepsilon_\F = -1$, view the corresponding $\F$-vector space as a vector space over $\F'$ via the inverse isomorphism $\iota_\F: \F \simeq \F'$. Then $\beta$ is covariant in the usual sense $\beta(t\cdot u, t\cdot v) = t\cdot \beta(u, v)$.
\end{proof}

\begin{remarks*}\
\begin{itemize}
\item
Tindzogho Ntsiri has another result \cite[\S5.1.2, Proposition 5.15 and Lemme 5.17]{TEtude} or \cite[Proposition 3.10]{TStructure}: assuming the characteristic (which is obviously common to $\K, \L, \M$, for instance by our method) is $p$ and $\kappa$, $\lambda$, $\mu$ are onto, one has $\K_+ \simeq \L_+ \simeq \M_+$ and $\K^\times \simeq \L^\times \simeq \M^\times$. We shall not use this.
\item
The author does not care for $n$-additive maps with $n \geq 3$.
\end{itemize}
\end{remarks*}

\subsection{Applied Tools: $\SL_2(\K)$-Modules}\label{s:applied}

We want some general results on $\SL_2(\K)$-modules of finite Morley rank; actually, part of the analysis requires less than the entire $\SL_2(\K)$-module structure and we shall try to clarify matters. This results in the notion of a non-trivial $T$-line, Definition \ref{d:line} below.

An important feature of our study is that we restrict to positive characteristic settings, as Fact \ref{f:char0} in \S\ref{s:conjecture} entirely settles the definable representation theory of algebraic groups in characteristic $0$.

From here on we adopt the following notations.

\begin{notation}\label{n:GUBTw}\
\begin{itemize}
\item
Let $\K$ be a definable field of positive characteristic $p$ and $G = \SL_2(\K)$;
\item
let $U$ be a maximal unipotent subgroup of $G$, $B = N_G(U)$ be the associated Borel subgroup, and $T \leq B$ be an algebraic torus;
\item
let $w \in N_G(T)\setminus T$ (an element of order at most $4$ inverting $T$).
\end{itemize}
Finally recall that $H$-module means: definable, connected, $H$-module.
\end{notation}

\begin{fact}[from {\cite[Lemma B]{BDRank3}}]\label{f:UVnilpotent}
If $V$ is a $U$-module, then the semi-direct product $U \ltimes V$ is nilpotent.
\end{fact}

(One may however observe that this could fail in characteristic $0$.)


\begin{notation}\label{n:Zi}
Let $V$ be a $U$-module.
\begin{itemize}
\item
Let $Z_0 = 0$ and $Z_{j+1}/Z_j = C_{V/Z_j}\o(U)$.

(As a consequence of Zilber's Indecomposibility Theorem, $[U, Z_{j+1}] \leq Z_j$.)
\item
Let $Y_j = Z_j/Z_{j-1}$.
\item
Let $\ell_U(V)$ be the least integer with $Z_n = V$. We call $\ell_U(V)$ the $U$-length, or unipotence length, of $V$.

(This is \emph{not} the same measure as in the proof of Fact \ref{f:Maschke}.)
\end{itemize}
\end{notation}

By Fact \ref{f:UVnilpotent}, the series $(Z_j)$ is strictly increasing until one reaches $Z_{\ell_U(V)} = V$.
Notice that if $V$ is actually a $B$-module, then so are $Z_j$ and $Y_j$.

\begin{lemma}[from {\cite[\S2.1, Step 7]{CDSmall}}]\label{l:bracket}
Let $V$ be a $B$-module. Then $[U, C_V(T)] \leq [T, V]$.
\end{lemma}
\begin{proof}
For brevity we let $C = C_V(T)$ and $D = [T, V]$.

Let $c_0 \in C$ and $u \in U\setminus\{1\}$. Since $V = C \oplus D$ by Fact \ref{f:semisimple}, there is $(c, d) \in C\times D$ with $[c_0, u] = c + d$. Now for any $t \in T$, $[c_0, t\cdot u] = c + t \cdot d \in \<c\> + D$. Since $T\cdot u$ generates $U$, one finds $[c_0, U] \leq \<c\> + D$; by Zilber's Indecomposibility Theorem and since the characteristic is finite, actually $[c_0, U] \leq D$, as desired.
\end{proof}

\begin{lemma}\label{l:noZiZi+1}
Let $V$ be a $B$-module of $U$-length exactly $n$ and $2 \leq j \leq n$. Then $T$ cannot centralise both $Y_j$ and $Y_{j-1}$ (see Notation \ref{n:Zi}). More precisely, if $T$ centralises $Y_{j-1}$ then $C_{Y_j}(T) = 0$.
\end{lemma}
\begin{proof}
Let $\beta : U \times Y_j \to Y_{j-1}$ be given by:
\[\beta(u, z + Z_{j-1}) = [u, z] + Z_{j-2}\]
It is routine to check that $\beta$ is well-defined, bi-additive, and $T$-covariant.

We claim that the right radical of $\beta$ (usually denoted $U^\perp$) is finite. For if $z + Z_{j-1}$ is $\beta$-orthogonal to $U$, then $[U, z] \leq Z_{j-2}$ and $z \in C_{V/Z_{j-2}}(U)$, the connected component of which is exactly $Z_{j-1}$. 

Now assuming that $T$ centralises $Y_{j-1}$, let $x \in C_{Y_j}(T)$ and $(u, t) \in U \times T$; one finds:
\[\beta(t\cdot u - u, x) = t\cdot \beta(u, x) - \beta (u, x) = 0\]
Since $[T, u] = U$, we get $\beta(U, x) = 0$ for arbitrary $x \in C_{Y_j}(T)$, forcing finiteness of the latter, and its triviality by connectedness of $Y_j$ and Fact \ref{f:semisimple}.
\end{proof}

\begin{lemma}\label{l:Bruhat}
Let $V$ be an irreducible $G$-module of $U$-length exactly $n$. Then $Z_1 \cap w\cdot Z_{n-1}$ is finite.
\end{lemma}
\begin{proof}
We claim that $V_0 = \<G\cdot (Z_1\cap w\cdot Z_{n-1})\> \leq Z_{n-1}$. For let $v \in Z_1 \cap w \cdot Z_{n-1}$ and $g \in G$; by the Bruhat decomposition $G = B \sqcup BwU$ there are two cases.
\begin{itemize}
\item
If $g \in B$, write $g = tu$ with obvious notations; then $g\cdot v = t \cdot v \in Z_1 \leq Z_{n-1}$.
\item
If $g \in BwU$ then $g = bwu$ and $g\cdot v = bw \cdot v \in Z_{n-1}$ again.
\end{itemize}
Since $Z_{n-1}$ is proper in the irreducible $G$-module $V$, the $G$-module $V_0$ is finite.
\end{proof}

\begin{lemma}\label{l:noZn}
Let $V$ be an irreducible $G$-module of $U$-length exactly $n$. Then $T$ does not centralise $Y_n = Z_n/Z_{n-1} = V/Z_{n-1}$.
\end{lemma}
\begin{proof}
Suppose it does. By Lemma \ref{l:Bruhat}, $w\cdot Z_1 \cap Z_{n-1}$ is finite, so for some finite $K$ the quotient $w\cdot Z_1/K$ embeds into $Y_n$ as a $T$-module. It follows that $T$ centralises $w\cdot Z_1/K$: by connectedness of $Z_1$, $T$ centralises $w\cdot Z_1$.
Also notice that by assumption and Fact \ref{f:semisimple}, $[T, V] \leq Z_{n-1}$.

Now let $u \in U$ be such that $(uw)^3 = 1$ (see Notation \ref{n:SL2} below if necessary), let $i$ stand for the central involution of $G$ (it is harmless to let $i = 1$ in characteristic $2$), and $a \in Z_1$. Then:
\begin{align*}
uwu \cdot a & = uw a = w a + [u, wa]\\
= (wuw)^{-1} \cdot a & = w u^{-1} w a = i a + w[u^{-1}, wa].
\end{align*}
Notice that $[u, wa] \in [U, Z_n] \leq Z_{n-1}$. Moreover, since $wa \in w \cdot Z_1 \leq C_V(T)$, one has $[u^{-1}, wa] \in [U, C_V(T)]\leq [T, V]$ by Lemma \ref{l:bracket}, so $w[u^{-1}, wa] \in w \cdot [T, V] = [T, V] \leq Z_{n-1}$ as observed.

As a consequence $wa = i a + w[u^{-1}, wa] - [u, wa] \in w \cdot Z_1 \cap Z_{n-1}$ which is finite by Lemma \ref{l:Bruhat}, against $Z_1$ being infinite.
\end{proof}

\begin{lemma}\label{l:Tlines}
Let $V$ be a $B$-module. For $j \neq 1$, $Y_j$ from Notation \ref{n:Zi} is the direct sum of $C_{Y_j}(T)$ and of non-central $T$-minimal modules of the same rank as $\K$.

If $V$ is actually an irreducible $G$-module, then the same holds of $j = 1$.
\end{lemma}
\begin{proof}
We already know from Fact \ref{f:semisimple} that $Y_j = C_{Y_j}(T) \oplus [T, Y_j]$, so by Maschke's Theorem (Fact \ref{f:Maschke}) the statement is about $T$-minimal submodules of the latter summand, if any. Let $L \leq [T, Y_j]$ be one such.

Take again the bi-additive, $T$-covariant function $\beta: U \times Y_j \to Y_{j-1}$ from Lemma \ref{l:noZiZi+1}; remember that the right radical is finite. So the restriction $\beta: U\times L \to Y_{j-1}$ is non-zero. Let $M$ be a $T$-minimal subquotient of $Y_{j-1}$ such that $\pi\circ \beta: U \times L \to M$ still remains non-zero.

Now using the action of $T$ on $L$, $M$ respectively and Zilber's Field Theorem, interpret fields $\L, \M$ (possibly $\M = \F_p$ if $T$ centralises $M$), then prepare to apply Theorem \ref{t:threefields}. Observe from the structure of Borel subgroups of $\SL_2(\K)$ that here $\ker \kappa = Z(G)$ has order at most $2$.

Suppose $\rk L \neq \rk \K$.
Notice that $\ker\o \lambda \neq 1$; this is clear if $\rk L < \rk \K$, and a consequence of Wagner's Torus Theorem \cite{WFields} if $\rk L > \rk \K$ as $T \simeq \K^\times$ cannot embed properly into $\L^\times$ in positive characteristic. Hence $\ker\o \lambda$ and $\ker\lambda/\ker \rho$ are infinite. But also $\im\o \lambda \neq 1$ since otherwise $T$ centralises $L$, against its definition. By Theorem \ref{t:threefields}, $\K \simeq \L$, against $\rk \L = \rk L \neq \rk \K$. This shows $\rk L = \rk \K$.

If $V$ is actually an irreducible $G$-module and $j = 1$, use Lemma \ref{l:Bruhat} to embed $w\cdot Z_1$ into $V/Z_{n-1} = Y_n$ as a $T$-module (up to a finite kernel).
\end{proof}

\begin{definition}\label{d:line}
A non-trivial $T$-line is a non-central, $T$-minimal module of same rank as $\K$.
\end{definition}

For instance, it is a consequence of Lemmas \ref{l:noZn} and \ref{l:Tlines} that $Z_n$ contains a non-trivial $T$-line not included in $Z_{n-1}$.

\begin{corollary}\label{c:Tlines}
If $V$ is an irreducible $G$-module, then $[T, V]$ is a direct sum of an even number of non-trivial $T$-lines.
\end{corollary}
\begin{proof}
Let us refine the non-maximal $B$-series $(Z_j)$ into a composition $T$-series: one gets $T$-central factors and non-trivial $T$-lines. The number of the latter must be even since $w$ normalises $[T, V]$ and a $w$-invariant, $T$-minimal module must be centralised by $T$, as  there are no definable groups of automorphisms of fields of finite Morley rank \cite[Theorem 8.3]{BNGroups}.

The careful reader will observe that one ought to pay more attention to factors in composition series than we just did: since we work in the connected category, uniqueness in a Jordan-H\"older property can hold only up to isogeny. But in the case of connected $T$-modules, this is harmless.
\end{proof}

\begin{remarks*}\
\begin{itemize}
\item
A weaker result had been obtained by Tindzogho Ntsiri \cite[\S5.2.2, Proposition 5.31]{TEtude} or \cite[Theorem 4.6]{TStructure} under the assumption that $\SL_2(\K)$ acts faithfully on $V$.
\item
These results are far from optimal. The main problem is that the centraliser $C_V(T)$ could be ``scattered'' among the various $Y_j$, i.e. $C_{Y_j}(T)$ could be non-zero for several indices, something prohibited in the algebraic setting. This for the moment seems to be the major obstruction to understanding the $B$-structure of $\SL_2(\K)$-modules of arbitrary rank, and we view it as a crucial investigation to lead in the near future.
\end{itemize}
\end{remarks*}

\section{Proof of the Main Theorem}\label{S:Proof}

Let us state our result again.

\begin{theorem*}
Definable in a universe of finite Morley rank, let $\K$ be a field, $G = \SL_2(\K)$, and $V$ be an irreducible $G$-module of rank $\leq 4\rk \K$.
Suppose the characteristic of $\K$ is neither $2$ nor $3$.

Then there is a definable $\K$-vector space structure on $V$ such that $V$, as a $\K[G]$-module, $V$ is definably isomorphic to one of the following:
\begin{itemize}
\item
$\Nat G$, the natural representation;
\item
$\Ad G \simeq \Sym^2 \Nat G$, the adjoint representation;
\item
$\Sym^3 \Nat G$, the rational representation of degree $4$ on $\K[X^3, X^2Y, XY^2, Y^3]$;
\item
$\Nat G \otimes{}^\chi (\Nat G)$, a twist-and-tensor representation for some definable $\chi \in \Aut(\K)$.
\end{itemize}
\end{theorem*}

Let us confess again that we do not know what happens in low characteristic.

\subsection{Notations and Facts}\label{s:notations}

The characteristic $0$ case is covered by Fact \ref{f:char0} from \S\ref{s:conjecture}, so let us turn to positive characteristics.
We keep and extend Notations \ref{n:GUBTw} (see \S\ref{s:applied}); as a matter of fact we need to coordinatise $G = \SL_2(\K)$.

\begin{notation}\label{n:SL2}
For $\lambda \in \K_+$ (resp., $\K^\times$), let:
\[u_\lambda = \begin{pmatrix} 1 & \lambda \\ & 1\end{pmatrix}, \quad t_\lambda = \begin{pmatrix} t_\lambda \\ & t_\lambda^{-1}\end{pmatrix}, \quad w = \begin{pmatrix} & 1\\-1\end{pmatrix}\]
We let $i = t_{-1}$ be the central involution of $G$ (by assumption, $p \neq 2$); of course $w^2 = i$.

Notice that letting $U = \{u_\lambda: \lambda \in \K_+\} \simeq \K_+$ and $T = \{t_\lambda: \lambda \in \K^\times\} \simeq \K^\times$ is consistent with Notation \ref{n:GUBTw}.
Also notice that $t_\mu u_\lambda t_\mu^{-1} = u_{\lambda\mu^2}$.

Finally check that:
\[u_\lambda w u_{\lambda^{-1}} w u_\lambda w = t_\lambda\]
For brevity let $u = u_1$; then the above relation simplifies into $(uw)^3 = 1$.
\end{notation}

\begin{notation}\label{n:V}
Let $V$ be an irreducible $G$-module of rank $\leq 4 \rk \K$. For brevity let $k = \rk (\K)$.
\end{notation}

\begin{notation}\label{n:d}
For $\lambda \in \K$, let $\d_\lambda = \ad_{u_\lambda} : V \to V$ map $v$ to $[u_\lambda, v] = u_\lambda v - v$. For brevity let $\d = \d_1$.
\end{notation}

Hence $\d_\lambda = u_\lambda - 1$ in $\End(V)$. Notice that $t_\mu \d_\lambda t_\mu^{-1} = \d_{\lambda\mu^2}$.

\paragraph{The rank $4k$ analysis starts here.}

The structure of the proof is rather straightforward. We first bound the unipotence length in Proposition \ref{p:length} (\S\ref{s:length}). Small rank cases, where $\rk V \leq 3k$, are known by \cite{DActions} and \cite{CDSmall}; we however give shorter proofs in \S\ref{s:3k}. Then real things begin. We cover $\Sym^3 \Nat G$ in Proposition \ref{p:quartic} (\S\ref{s:quartic}), and then $\Nat G\otimes{}^\chi (\Nat G)$ in Proposition \ref{p:twist} (\S\ref{s:twist}).

We shall rely on a non-model-theoretic investigation of $\Sym^n \Nat G$. In the following, composition series are in the classical sense (no definability nor connectedness required).

\begin{fact}[{\cite[Theorem 1]{TV-III}}]\label{f:theorem1}
Let $V$ be a $\Q[\SL_2(\Z)]$-module. Suppose that for every unipotent element $u \in \SL_2(\Z)$, $(u-1)^5 = 0$ in $\End(V)$. Then $V$ has a composition series each factor of which is a direct sum of copies of $\Q \otimes_\Z \Sym^k \Nat \SL_2(\Z)$ for $k \in \{0, \dots, 4\}$.
\end{fact}

Although Fact \ref{f:theorem1} is stated for $\SL_2(\Z)$-modules, its highly computational proof readily adapts to $\F_p[\SL_2(\F_p)]$-modules, provided $p$ is greater than the unipotence length; incidently the author does not know whether a character-theoretic proof exists in that case.
One should however refrain from using Fact \ref{f:theorem1}.
As indicated by \cite[Proposition 2]{TV-III}, there are deep geometric reasons which prevent the computation from extending to length $\geq 7$, at least for $\Q[\SL_2(\Z)]$-modules. So we are dissatisfied with this result, albeit next-to-optimal.

A result we are satisfied with is the following.

\begin{fact}[{\cite[Theorem 2]{TV-III}}]\label{f:theorem2}
Let $n \geq 2$ be an integer and $\K$ be a field of characteristic $0$ or $\geq 2n+1$. Suppose that $\K$ is $2(n-1)!!$-radically closed. Let $G = \SL_2(\K)$ and $V$ be a $G$-module. Let $\K_1$ be the prime subfield and $G_1 = \SL_2(\K_1)$. Suppose that $V$ is a $\K_1$-vector space such that $V \simeq \oplus_I \Sym^{n-1} \Nat G_1$ as $\K_1[G_1]$-modules.

Then $V$ bears a compatible $\K$-vector space structure for which one has $V \simeq \oplus_J \Sym^{n-1} \Nat G$ as $\K[G]$-modules.
\end{fact}

Moreover, the linear structure is definable over $G$ by construction.

\subsection{Bounding the Length}\label{s:length}

\begin{proposition}\label{p:length}
$V$ as in Notation \ref{n:V} has $U$-length at most $4$.
\end{proposition}
\begin{proof}
If there are four non-trivial $T$-lines (see Definition \ref{d:line}) we are done for rank reasons. So suppose there are only two. By Lemma \ref{l:noZn} at least one occurs in the final factor $Y_n$; by Lemma \ref{l:noZiZi+1} the other one cannot occur more than two steps afar, and cannot be more than two steps away from $Z_0$. The worst case is therefore where the two lines occur at stages $2$ and $4$: $\ell_U(V) \leq 4$.
\end{proof}

\subsection{Warming up: Shorter Proofs in Rank 3k}\label{s:3k}

Since one can also measure technical progress in a topic by the ability to produce shorter proofs of existing theorems, we shall now revisit two earlier results.

\begin{fact}[rank $2k$ analysis: {\cite[Theorem B]{DActions}}]\label{f:2k}
If $V$ as in Notation \ref{n:V} has rank $\leq 2k$, then there is a definable $\K$-vector space structure such that $V\simeq \Nat G$.
\end{fact}
The proof will appear only marginally shorter than the original argument, but it is self-contained as opposed to \cite{DActions} which relied on work by Timmesfeld \cite[Theorem 3.4]{TAbstract}. (From the latter we extract only the core computation, the group $\SL_2(\K)$ being already identified and coordinatised.)
\begin{proof}
For rank reasons, it follows from Corollary \ref{c:Tlines} that $[T, V]$ is the direct sum of exactly two non-trivial $T$-lines. As a consequence, $V = Z_2$ has $U$-length $2$ and $C_V(T) = 0$.

Now let $a \in Z_1$; with a look at Notation \ref{n:d}, set $b = \d w a = [u, wa] \in [U, V] \leq Z_1$, and $c = \d wb \in Z_1$. One finds:
\begin{align*}
(uw)^2 \cdot a & = uw (wa + b) = ia + wb + c\\
= (uw)^{-1} \cdot a & = iw u^{-1} a = iw a
\end{align*}
so $w(ia - b) = ia + c \in Z_1 \cap w \cdot Z_1 \leq C_V(U, wUw^{-1}) = C_V(G) \leq C_V(T) = 0$.
Hence $b = ia$.

Now for arbitrary $\lambda \in \K^\times$, one still has $t_\lambda a \in Z_1$ so:
\[u_{\lambda^2} w a = t_\lambda u t_{\lambda}^{-1} w a = t_\lambda u w t_\lambda a = t_\lambda (w t_\lambda a + i t_\lambda a) = wa + i t_{\lambda^2} a\]
and since $\K$ is algebraically closed, one finds $u_\lambda wa = w a + i t_\lambda a$.

We may now define the vector space structure. On $Z_1$ let $\lambda \cdot a = t_\lambda a$; on $w\cdot Z_1$ let $\lambda \cdot (wa) = w (\lambda \cdot a)$.
This is obviously additive in $a$ and multiplicative in $\lambda$; moreover, since $U$ centralises $[U, V] \leq Z_1$,
\begin{align*}
(\lambda + \mu) \cdot a & = t_{\lambda + \mu} a = i [u_{\lambda + \mu}, wa] = i (u_\lambda \cdot [u_\mu, wa] + [u_\lambda, wa]) = i^2 (t_\mu a + t_\lambda a)\\
& = \lambda \cdot a + \mu \cdot a
\end{align*}
so this does define a linear structure on $Z_1 \oplus w \cdot Z_1$. (We already observed that the sum is direct indeed.)

We claim that $Z_1 \oplus w \cdot Z_1$ is then a $\K[G]$-module. First observe that it is $\<w, U\> = G$-invariant alright. By construction $w$ acts linearly. Linearity of $U$ is obvious on $Z_1$. Now for $(\lambda, \mu) \in \K^2$ and $a \in Z_1$, one has:
\begin{align*}
u_\mu (\lambda \cdot wa) & = u_\mu w(\lambda \cdot a) = w (\lambda \cdot a) + i t_\mu t_\lambda a = \lambda \cdot wa + \lambda \cdot (it_\mu a) = \lambda \cdot u_\mu a
\end{align*}
which proves linearity of $U$ on $w \cdot Z_1$ too.

By irreducibility, $V = Z_1 \oplus w\cdot Z_1$ which is clearly isomorphic to $\Nat G$.
\end{proof}

\begin{remarks*}\label{r:quadratic}\
\begin{itemize}
\item
The argument applies to an irreducible $G$-module of $U$-length exactly $2$ such that $C_V(T) = 0$, without mentioning the rank nor even its finiteness: this is a special case of Facts \ref{f:theorem1} and \ref{f:theorem2}.
\item
In the case of an irreducible $G$-module of $U$-length exactly $2$ such that $C_V(T) = 0$, there are no restrictions on the characteristic, which may be $0$, $2$ or $3$ in the above computation.
\item
In order to remove from the latter the $C_V(T) = 0$ assumption in a finite Morley rank setting, one can follow \cite[Claim 5.7]{DActions} or its generalisation to $\SL_n(\K)$ \cite[Proposition 2.1, Claim 1]{BDRank3}.
\end{itemize}
\end{remarks*}

\begin{fact}[rank $3k$ analysis: {\cite{CDSmall}}]\label{f:3k}
If $V$ as in Notation \ref{n:V} has rank $\leq 3k$, then there is a definable $\K$-vector space structure such that either $V \simeq \Nat G$ or $V \simeq \Sym^2 \Nat G$ (the adjoint representation).
\end{fact}
\begin{proof}
We may assume $\rk V > 2k$; then $\ell_U(V) \geq 3$ since otherwise $V/C_V(G)$ has rank $2k$, and $V$ has too by finiteness of $C_V(G)$. Here again, $[T, V]$ is the direct sum of two non-trivial $T$-lines; it follows that $C_V(T)$ is infinite; in particular, the central involution $i$ acts as $1$.

Notice that $[U, C_V(T)] \neq 0$ as otherwise $U$ centralises $C_V(T) = w\cdot C_V(T)$, so $G = \<U, wUw^{-1}\>$ centralises $C_V(T)$; by irreducibility the latter is therefore finite (notice that we did not reduce to $C_V(G) = 0$), a contradiction.

Hence $[U, C_V(T)] \neq 0$; on the other hand equality does not hold in the inclusion $[U, C_V(T)] \leq [T, V]$ given by Lemma \ref{l:bracket}: as otherwise $[T, V]$ is $\<U, w\> = G$-invariant, so by irreducibility again, $V = [T, V]$, against $C_V(T) \neq 0$.

The conclusion of the above analysis is that $[U, C_V(T)]$ consists in exactly one non-trivial $T$-line $L$; it follows that $[T, V] = L \oplus w\cdot L$. Since $L$ is $U$-invariant and $T$-minimal, it is centralised by $U$; hence $L \leq Z_1$.


%

By Lemma \ref{l:noZn} there is a non-trivial $T$-line in $Y_n$, where $n > 2$ is the unipotence length, and in particular there is none in $Y_{n-1}, \dots, Y_2$: so by Lemma \ref{l:noZiZi+1}, $n -1 = 2$; $V$ is a cubic module in the sense $[U, U, U, V] = 0$. Let us study its structure further. First we see that $T$ centralises $Y_2$.

A priori, by Fact \ref{f:semisimple}, one should expect $Z_1 = L \oplus X_1$ with $X_1 = C_{Z_1}(T)$. But by Lemma \ref{l:Bruhat}, $w\cdot X_1$ embeds (modulo a finite kernel) into $Y_3$: both $w\cdot X_1$ and $Y_2$ are centralised by $T$, which violates Lemma \ref{l:noZiZi+1}. So $X_1 = 0$ and $Z_1 = L = [U, C_V(T)]$.
Likewise, $[T, Y_3] = Y_3$ is a non-trivial $T$-line. Rank reasons then yield $V = L \oplus C_V(T) \oplus w \cdot L$.

It is time to identify the action. Over the prime field $\F_p$ one could rely on Fact \ref{f:theorem1} but this is not the spirit, so we perform a subset of the mentioned computations.

In view of the decomposition $V = L \oplus C_V(T) \oplus w \cdot L$, there are additive maps $c: L \to C_V(T)$ and $\ell: L \to L$ such that for any $a\in L = Z_1$, one has $uw a = wa + c(a) + \ell(a)$. Recall from Notation \ref{n:d} that $\d$ stands for $\ad_u = u-1 \in \End(V)$. Hence:
\begin{align}
uw uw a & = uw (wa + c(a) + \ell(a))\nonumber\\
& = a + w c(a) + \d w c(a) + w \ell(a) + c(\ell(a)) + \ell(\ell(a))\nonumber\\
= (uw)^{-1} a & = w a\nonumber
\end{align}
Projecting onto $w \cdot L$, one finds $\ell(a) = a$; projecting onto $C_V(T)$, $w c(a) = - c (\ell(a)) = - c(a)$.

If $c(a) = 0$ for $a \neq 0$, then actually $uw a = w a + a$ and therefore $wa = (uw)^{-1} a = (uw)^2 a = uw (wa + a) =  a +  w a + a$, a contradiction to the characteristic not being $2$. So $c: L  \to C_V(T)$ is injective, and therefore $\rk C_V(T) \geq k$; $c$ is also surjective, so the previous computation shows that $w$ inverts $C_V(T)$.

Now let $d(a) = \d c(a) \in [U, C_V(T)] = L$. Find:
\begin{align*}
uwuw c(a) & = - uw (c(a) + d(a)) = c(a) + d(a) - w d(a) - c(d(a)) - d(a)\\
= (uw)^{-1} c(a) & = w c(a) - wd(a) = - c(a) - w d(a)
\end{align*}
so $c(d(a)) = 2 c(a)$. Since the map $c: L \to C_V(T)$ is injective, $d(a) = 2a$.

It is now clear that $V$ is a direct sum of copies of $\Sym^2 \Nat G_1$, where $G_1 = \<u, w\> = \SL_2(\F_p)$.
We then apply Fact \ref{f:theorem2}. Since the resulting $\K$-linear structure is definable, there is a definable isomorphism with $\Sym^2 \Nat G$.
(After reading again the end of the argument of \cite{CDSmall}, we have no clue what Cherlin meant that day.)
\end{proof}

\begin{remarks*}\
\begin{itemize}
\item
Here we did assume that the characteristic was neither $0$ (at the beginning of the argument, when we used Lemma \ref{l:bracket}) nor $2$ (at the end of the argument, for our computations); it could however be $3$.
\item
The characteristic $2$ case is dealt with quickly in \cite{CDSmall}; remember that the characteristic $0$ case is most efficiently handled by the model theory of Fact \ref{f:char0}.
\item
The argument does not classify all cubic $\SL_2(\K)$-modules of finite Morley rank (if it did there would be no Proposition \ref{p:twist}). What it does classify is those cubic modules with descending coherence degree $\kappa(V)$ (\cite[\S5]{TV-II}) equal to $1$, i.e. with $Z_1 = \ker \d$; there are no assumptions on the rank. The interested reader should try to prove directly that $\ker \d = Z_1$ in the $\rk V \leq 3k$ configuration; even our quickest argument goes through splitting $V = L \oplus C_V(T) \oplus w \cdot L$.
\end{itemize}
\end{remarks*}

\subsection{The Quartic Analysis}\label{s:quartic}

\begin{proposition}\label{p:quartic}
If $V$ as in Notation \ref{n:V} has $U$-length exactly $4$, then there is a definable $\K$-vector space structure such that $V \simeq \Sym^3 \Nat G$.
\end{proposition}

Before the proof it could be good to remind the reader that $\Sym^3 \Nat G$ is naturally isomorphic to the representation of $\SL_2(\K)$ in the $4$-dimensional space $\K[X^3, X^2Y, XY^2, Y^3]$ of homogeneous polynomials in $X, Y$ with degree $3$, and to encourage him to perform a few calculations there (exponentiating the action of the Lie algebra is an alternative to direct computation).

\begin{proof}[Proof of Proposition \ref{p:quartic}]
\setcounter{step}{0}
\begin{notationinproof}
For $j = 1\dots 4$, let $\check{Z}_j = (Z_j \cap w\cdot Z_{n+1-j})\o$ (here $n = 4$).
\end{notationinproof}

\begin{step}\label{p:quartic:st:decomposition}
$V = \oplus_{j = 1}^4 \check{Z}_j$.
\end{step}
\begin{proofclaim}
If $V = [T, V]$, the claim follows from \ref{c:Tlines}: in that case and for a rank reason, $V$ is a direct sum of exactly four non-trivial $T$-lines, and since the $U$-length is exactly four, one sees that each $Y_i$ is a non-trivial $T$-line. Then intersections are easily taken.

Notice that $V = [T, V]$ would follow from the next Claim \ref{p:quartic:st:i} --- but we do not know the action of the central involution yet.

So let us assume towards a contradiction that $C_V(T) \neq 0$. Now there are exactly two non-trivial $T$-lines. Since $T$ does not centralise $Y_4$ by Lemma \ref{l:noZn}, and in view of Lemma \ref{l:noZiZi+1}, these lines are in $Y_2$ and $Y_4$ (an argument already used for Proposition \ref{p:length}). At this point $T$ centralises $Y_1$ and $Y_3$, so by Lemma \ref{l:noZiZi+1} again, $[T, Y_4] = Y_4$ is a non-trivial $T$-line.
On the other hand, by Lemma \ref{l:Bruhat}, $w \cdot Z_1$ embeds modulo a finite kernel into $Y_4$, so $w\cdot Z_1$ is a non-trivial $T$-line: $Z_1$ as well, a contradiction.
\end{proofclaim}

\begin{step}\label{p:quartic:st:i}
The central involution $i$ acts as $-1$.
\end{step}
\begin{proofclaim}
The awful computation underlying Fact \ref{f:theorem1} (more specifically, \cite[\S1.2.7]{TV-III}) is not fully conclusive since one could imagine a globally quartic $\PSL_2(\Z)$-module which splits into cubic-by-cubic. So let us compute.

Suppose $i$ acts as $1$; fix $a \in Z_1$. Noticing $w\cdot \check{Z}_2 = \check{Z}_3$ and given the decomposition of Claim \ref{p:quartic:st:decomposition}, there are additive maps $\alpha: Z_1 \to Z_1$, and $\beta, \gamma : Z_1 \to \check{Z}_2$ such that:
\[
uw a = w a + w \gamma(a) + \beta(a) + \alpha(a)
\]
Hence:
\begin{align*}
uw uw a & = a + \gamma(a) + \d\gamma(a) + w \beta(a) + \d w \beta(a)\\
& \quad  + w\alpha(a) + w \gamma(\alpha(a)) + \beta(\alpha(a)) + \alpha(\alpha(a))\\
 = (uw)^{-1} a & = w a
\end{align*}
Projecting onto $\check{Z}_4$ one finds $\alpha(a) = a$, then projecting onto $\check{Z}_3$ one finds $\gamma(a) = - \beta(a)$, so finally $ \d w \beta(a) = \d \beta(a) - 2 a$.

Then letting $\delta(a) = \d \beta(a) - 2 a \in Z_1$:
\begin{align*}
uw uw \beta(a) & = uw (w \beta(a) + \delta(a))\\
& = \beta(a) + \d\beta(a) + w \delta(a) - w \beta(\delta(a)) + \beta(\delta(a)) + \delta(a)\\
= (uw)^{-1} \beta(a) & = w \beta(a) - w \d \beta(a)
\end{align*}
Now this yields $\delta(a) = - \d\beta(a) = \d\beta(a) - 2 a$, and $\d\beta(a) = a = - \delta(a)$. Hence $\d w a = - w \beta(a) + \beta(a) + a$ entails $\d^2 wa = 0$.

We just proved $\d^2(w\cdot Z_1) = 0$; conjugating by $T$, one has $\d_\lambda^2(w\cdot Z_1) = 0$ for any $\lambda \in \K$; since the characteristic is not $2$, $\d_\lambda \d_\mu (w \cdot Z_1) = 0$ for any $(\lambda, \mu) \in \K^2$. This means $w \cdot Z_1 \leq Z_2$. Now Lemma \ref{l:Bruhat} forces the $U$-length to be $2$, a contradiction.
\end{proofclaim}

\begin{step}\label{p:quartic:st:beta}
For $\lambda \in \K^\times$ there is an abelian group homomorphism $\beta_\lambda: Z_1 \to \check{Z}_2$ such that for any $a \in Z_1$:
\[
u_\lambda w a = w a + w t_\lambda^{-1} \beta_\lambda(a) + \beta_\lambda (a) - t_\lambda a
\]
As a matter of fact, $\beta_{\lambda^2}(a) = t_\lambda \beta_1(t_\lambda a)$.
\end{step}
\begin{proofclaim}
We shall first do it over the prime field, i.e. for $\lambda = 1$, and vary the field and values of $\lambda$ afterwards.
So for the moment set $\lambda = 1$; this is a quick computation, very similar to Claim \ref{p:quartic:st:i} (as a matter of fact, one could reach both Claims simultaneously but the exposition may not benefit from doing so).

Fix $a \in Z_1$. Noticing $w\cdot \check{Z}_2 = \check{Z}_3$ and given the decomposition of Claim \ref{p:quartic:st:decomposition}, one should expect additive maps $\alpha: Z_1 \to Z_1$, and $\beta, \gamma : Z_1 \to \check{Z}_2$ such that:
\[
uw a = wa + w \gamma(a) + \beta (a) + \alpha(a)
\]
Bearing in mind that $(uw)^3 = 1$, one finds:
\begin{align*}
uw uw \cdot a & = u (- a - \gamma(a) + w \beta(a) + w \alpha(a))\\
& = - a - \gamma(a) - \d \gamma(a) + w \beta(a) + \d w \beta(a)\\
& \quad + w \alpha(a) + w \gamma(\alpha(a)) + \beta(\alpha(a)) + \alpha (\alpha(a))\\
= (uw)^{-1} a & = - wu^{-1} a = - wa
\end{align*}
Projecting onto $\check{Z}_4 = w \cdot Z_1$, this yields $\alpha(a) = - a$; projecting onto $\check{Z}_3$, one finds $\gamma(a) = - \beta(\alpha(a)) = \beta(a)$.

Now let $\lambda \in \K^\times$ and $\ell \in \K^\times$ be a square root. Then:
\begin{align*}
u_\lambda w a & = t_\ell u t_\ell^{-1} w a = t_\ell u w t_\ell a\\
& = t_\ell (w t_\ell a + w \beta(t_\ell a) + \beta(t_\ell a) - t_\ell a)\\
& = w a + t_\ell w \beta(t_\ell a) + t_\ell \beta(t_\ell a) - t_\lambda a\\
& = w a + w t_\lambda^{-1} t_\ell \beta (t_\ell a) + t_\ell \beta(t_\ell a) - t_\lambda a
\end{align*}
So one should let $\beta_\lambda(a) = t_\ell \beta(t_\ell a)$ to find the desired decomposition.
\end{proofclaim}

Model theory, more specifically connectedness arguments, will play a role in the next Claim. Without finiteness of the rank, it might fail (the brave should see \cite[\S1.2.8]{TV-III}).

\begin{step}\label{p:quartic:st:beta:isomorphism}
$\beta_\lambda: Z_1 \to \check{Z}_2$ is a group isomorphism.
\end{step}
\begin{proofclaim}
%
%
%
Let $a \in Z_1$, and introduce the definable sets $A = \{\mu \in \K^\times: \beta_\mu (a) = 0\} \cup\{0\}$ and $X = \K_+ \setminus A$. Suppose $A \neq 0$ and fix some $\lambda \in A\setminus\{0\}$, ie. with $\beta_\lambda(a) = 0$.

Thanks to Claim \ref{p:quartic:st:beta}, one has for arbitrary $\mu \in \K^\times$:
\begin{align*}
u_\mu u_\lambda w a & = u_\mu (w a - t_\lambda a)\\
& = w a + w t_\mu^{-1} \beta_\mu(a) + \beta_\mu(a) - t_\mu (a) - t_\lambda(a)\\
= u_{\lambda + \mu} w a & = w a + w t_{\lambda+\mu}^{-1} \beta_{\lambda + \mu} (a) + \beta_{\lambda + \mu}(a) - t_{\lambda + \mu} a
\end{align*}
Projecting onto $\check{Z}_2$, one finds $\beta_{\lambda + \mu} (a) = \beta_\mu(a)$. Hence $A$ is a subroup of $\K_+$.

Let $\mu \in X$, i.e. with $\beta_\mu(a) \neq 0$, if there is such a thing.
Returning to the equation and now projecting onto $\check{Z}_3$, one finds $t_\mu^{-1} \beta_\mu(a) = t_{\lambda + \mu}^{-1} \beta_{\lambda + \mu}(a) = t_{\lambda + \mu}^{-1} \beta_\mu(a)$. Hence $t_{(\lambda + \mu)\mu^{-1}}$ centralises $\beta_\mu(a) \neq 0$.
Then by $T$-minimality of $\check{Z}_2$ and Zilber's Field Theorem, $t_{(\lambda + \mu)\mu^{-1}}$ centralises all of $\check{Z}_2$.

Therefore $X \subseteq \{\mu \in \K^\times: t_{(\lambda + \mu)\mu^{-1}} \in C_T(\check{Z}_2)\}$.
But $C_T(\check{Z}_2)$ is proper (hence not of maximal rank) in the connected group $T$, and the map $\mu \mapsto 1 + \frac{\lambda}{\mu}$ is injective on $X$, so $X$ is not of maximal rank in $\K_+$. Hence $A$ is of maximal rank, but it also is a subgroup.

This proves $A = \K_+$ \cite[Corollary 5.13 if not obvious]{BNGroups}, and $\beta_\mu(a) = 0$ for any $\mu \in \K^\times$. In particular,
\[u_\mu w a = w a - t_\mu a\]
for all $\mu \in \K^\times$; now $\<G \cdot a\>$ is easily seen to be a quadratic module, hence by irreducibility, $a = 0$. A rank argument then proves the claim.
\end{proofclaim}

\begin{step}\label{p:quartic:st:dbeta}
For $a \in Z_1$, one has $\d \beta(a) = - 3 a$.
\end{step}
\begin{proofclaim}
Return to the decomposition obtained in Claim \ref{p:quartic:st:beta}: letting $\beta = \beta_1$, one has $u w a = w a + w \beta(a) + \beta(a) - a$.
Then:
\begin{align*}
uw uw a & = uw ( w a + w \beta (a) + \beta(a) - a)\\
& = - a - \beta(a) - \d \beta(a) + w \beta(a) + \d w \beta(a)\\
& \quad - w a - w \beta(a) - \beta(a) + a\\
= (uw)^{-1} a & = - w u^{-1} a = - w a
\end{align*}
So that $\d w \beta(a) = 2 \beta(a) + \d \beta(a)$. Now let $q = \d \beta(a) \in Z_1$. Then:
\begin{align*}
uw uw \beta(a) & = uw (w \beta(a) + 2 \beta(a) + q)\\
& = - \beta(a) - q + 2 w \beta(a) + 4 \beta(a) + 2 q + w q + w \beta(q) + \beta(q) - q\\
& = w q + w \beta(q + 2a) + \beta(q + 3a)\\
= (uw)^{-1} \beta (a) & = - w (\beta(a) - q) = w q - w \beta(a)
\end{align*}
This forces $\beta(q + 3a) = 0$, so by Claim \ref{p:quartic:st:beta:isomorphism}, $q = - 3a$, as desired.
\end{proofclaim}

\begin{step}
There is a $\K$-vector space structure on $V$ for which $V \simeq \Sym^3 \Nat G$.
\end{step}
\begin{proofclaim}
At this stage we could rely on Fact \ref{f:theorem1} to determine the structure of $V$ as a $G_1$-module, where $G_1 = \SL_2(\F_p)$.
The argument would be as follows.

First, we compute the descending coherence degree $\kappa(V)$ in the sense of \cite[\S5]{TV-II}; more precisely we show $\kappa(V) = 1$, viz. $\ker \d = Z_1$. One inclusion is clear; for the other, let $v \in \ker \d$. In view of Claim \ref{p:quartic:st:decomposition}, we may suppose $v = w a + w b + c$ for $a \in Z_1$ and $(b, c) \in \check{Z}_2^2$. It is not hard to see that $\d w b = - \beta (\d b) + \d b$, so that $0 = \d v = w \beta(a) + \beta(a) - a - \beta (\d b) + \d b + \d c$. Projecting onto $\check{Z}_3$, one has $\beta(a) = 0$ whence $a = 0$ by Claim \ref{p:quartic:st:beta:isomorphism}; then projecting onto $\check{Z}_2$, $\beta(\d b) = 0$ whence $\d b = 0$; finally $\d c = 0$. But since $\beta: Z_1 \to \check{Z}_2$ is onto by Claim \ref{p:quartic:st:dbeta}, one sees that $\beta(\d b) = - 3b = 0$, so $b = c = 0$, and $v = 0$. This proves $\ker \d = Z_1$.

Now by Fact \ref{f:theorem1}, more specifically \cite[\S1.2.8]{TV-III} (``$n = 4, i = -1$'') there, as a $G_1$-module, $V$ has a composition series (in the classical sense of the term) $0 \leq V_1 \leq V_2 \leq V$ with $V_1$ and $V_2/V_1$ isomorphic to sums of copies of $\Nat G_1$, and $V/V_2$ a sum of copies of $\Sym^3 \Nat G_1$. However, the quadratic radical over $G_1$ is $G_1$-generated by $\ker \d \cap w \cdot \ker \d^2 = Z_1 \cap w \cdot Z_2$, and by Lemma \ref{l:Bruhat}, the latter is contained in $C_V(T) = 0$, so $V_2 = 0$ and $V \simeq V/V_2$ has the desired structure under $G_1$.


As a matter of fact we do not need \cite[Theorem 1]{TV-III} since we already know that for $a \in Z_1$, hold $u w a = w a + w \beta(a) + \beta(a) - a$ and $\d \beta(a) = -3a$: from there and the decomposition of Claim \ref{p:quartic:st:decomposition}, constructing a $G_1$-isomorphism with $\Sym^3 \Nat G_1$ is actually immediate.

In order to recover a $\K$-linear structure, we now use Fact \ref{f:theorem2} (more accurately, \cite[Theorem 2]{TV-III} and the remark following it to justify working in characteristic $\neq 2, 3$).
\end{proofclaim}

This completes the proof of Proposition \ref{p:quartic}.
\end{proof}

\subsection{The Tensor Analysis}\label{s:twist}

And now let us do the twist.

\begin{proposition}\label{p:twist}
If $V$ as in Notation \ref{n:V} has rank $>3k$ and $U$-length $\leq 3$, then there are a definable $\K$-vector space structure and a definable automorphism $\chi$ of $\K$ such that $V \simeq \Nat G\otimes{}^\chi(\Nat G)$ (twist-and-tensor).
\end{proposition}

Before the proof begins it could be useful to work out the details of the target module.

\begin{example}\label{ex:twist}
Let $\varphi, \psi \in \Aut(\K)$. Let $\K^2$ have basis:
\[e_1 = \begin{pmatrix} 1 \\ 0 \end{pmatrix}, \quad e_2 = \begin{pmatrix} 0 \\ 1\end{pmatrix}\]
Then in the twisted representation $^\varphi(\Nat G)$, one has:
\[
\left\{
\begin{array}{rcl}
w \cdot e_1 & = & - e_2\\
w \cdot e_2 & = & e_1
\end{array}
\right.
,
\left\{
\begin{array}{rcl}
t_\lambda \cdot e_1 & = & \varphi(\lambda) e_1\\
t_\lambda \cdot e_2 & = & \varphi(\lambda)^{-1} e_2
\end{array}
\right.
,
\left\{
\begin{array}{rcl}
u_\lambda \cdot e_1 & = & e_1\\
u_\lambda \cdot e_2 & = & \varphi(\lambda) e_1 + e_2
\end{array}
\right.
\]
Now $^\varphi(\Nat G) \otimes{}^\psi (\Nat G)$ has basis $(e_1\otimes e_1, e_1 \otimes e_2, e_2\otimes e_1, e_2\otimes e_2)$, and the action there is given by:
\[
\left\{
\begin{array}{rcl}
w \cdot e_1\otimes e_1 & = & e_2\otimes e_2\\
w \cdot e_1\otimes e_2 & = & - e_2\otimes e_1\\
w \cdot e_2\otimes e_1 & = & - e_1\otimes e_2\\
w \cdot e_2\otimes e_2 & = & e_1 \otimes e_1
\end{array}
\right.
,
\left\{
\begin{array}{rcl}
t_\lambda \cdot e_1\otimes e_1 & = & \varphi(\lambda) \psi(\lambda) e_1\otimes e_1\\
t_\lambda \cdot e_1\otimes e_2 & = & \varphi(\lambda) \psi(\lambda)^{-1} e_1\otimes e_2\\
t_\lambda \cdot e_2\otimes e_1 & = & \varphi(\lambda)^{-1} \psi(\lambda) e_2\otimes e_1\\
t_\lambda \cdot e_2\otimes e_2 & = & \varphi(\lambda)^{-1} \psi(\lambda)^{-1} e_2\otimes e_2\\
\end{array}
\right.
,\]
\[
\left\{
\begin{array}{rcl}
u_\lambda \cdot e_1\otimes e_1 & = & e_1\otimes e_1\\
u_\lambda \cdot e_1\otimes e_2 & = & \psi(\lambda) e_1\otimes e_1 + e_1 \otimes e_2\\
u_\lambda \cdot e_2\otimes e_1 & = & \varphi(\lambda) e_1\otimes e_1 + e_2\otimes e_1\\
u_\lambda \cdot e_2\otimes e_2 & = & \varphi(\lambda) \psi(\lambda) e_1 \otimes e_1 + \varphi(\lambda) e_1\otimes e_2 + \psi(\lambda) e_2\otimes e_1 + e_2\otimes e_2
\end{array}
\right.
\]
The reader should keep these in mind during the forthcoming proof.
\end{example}

\begin{proof}[Proof of Proposition \ref{p:twist}]
\setcounter{theorem}{6}
\setcounter{step}{0}

\begin{step}\label{p:twist:st:i}
The central involution acts as $1$.
\end{step}
\begin{proofclaim}
The value of the central involution could be read from \cite{TV-III}; we shall argue more directly. Suppose $i = -1$ and let $a_1 \in Z_1$. Let $b_2 = \d w a_1 \in [U, V] \leq Z_2$. Then:
\[
uwuw\cdot a_1 = uw (wa_1 + b_2) = - a_1 + w b_2 + \d w b_2 = (uw)^{-1}\cdot a_1 = - w a_1
\]
This suggests to write $b_2$ as $q_2 - a_1$, with $q_2 \in Z_2$. Hence:
\[- wa_1 = - a_1 + w q_2 - w a_1 + \d w q_2 - (q_2 - a_1)\]
and therefore $uw q_2 = q_2$; in particular $w q_2 \in Z_2$. Now let $r_1 = \d q_2 \in Z_1$; observe how $q_2 = (uw)^{-1} q_2 = - w (q_2 - r_1) = w r_1 - w q_2$, so $w r_1 = (1+w) q_2 \in w \cdot Z_1 \cap Z_2$. The latter is finite by Lemma \ref{l:Bruhat}; since $1+w$ is a bijection, there are finitely many possible values for $q_2$. But the map $a_1 \mapsto q_2$ is additive from the connected group $Z_1$, so $q_2 = 0$.

This means that for any $a_1 \in Z_1$ one has $\d w a_1 = - a_1 \in Z_1$, and therefore $\d^2 w \cdot Z_1 = 0$. We argue as in the end of Claim \ref{p:quartic:st:i} of Proposition \ref{p:quartic} to reach a contradiction.
\end{proofclaim}

\begin{step}\label{p:twist:st:decomposition}
$V = Z_1 \oplus L \oplus w\cdot L\oplus w\cdot Z_1$, where $Z_1$ and $L$ are non-trivial $T$-lines, and $L \leq Z_2$.
\end{step}
\begin{proofclaim}
We shall first show that $Z_2\cap w \cdot Z_2$ is infinite by an argument similar to Claim \ref{p:twist:st:i}. Suppose $Z_2 \cap w\cdot Z_2$ is finite. Then for $a_1 \in Z_1$, letting $b_2 = \d w a_1 \in Z_2$ and $c_2 = \d w b_2 \in Z_2$:
\[
uwuw a_1 = a_1 + w b_2 + c_2 = w a_1
\]
so $a_1 - b_2 \in Z_2 \cap w\cdot Z_2$ which we assumed finite. Here again, $a_1 \mapsto a_1 - b_2$ is additive, so $b_2 = a_1$; then $uw a_1 = w a_1 + a_1$, implying $a_1 = (uw)^3 a_1 = 3 a_1 + 2 w a_1$, and since the characteristic is not $2$, $w a_1 = - a_1 \in Z_1 \cap w \cdot Z_1$ which is finite by Lemma \ref{l:Bruhat}, against $Z_1$ being infinite. So $Z_2 \cap w\cdot Z_2$ is infinite.

Now suppose that $C_V(T) \neq 0$. Notice that this reduces the number of non-trivial $T$-lines to exactly $2$. Since $T$ does not centralise $Y_3$ by Lemma \ref{l:noZn}, not both these lines are in $Z_2$; so by Lemma \ref{l:noZiZi+1} there is exactly one non-trivial $T$-line contained in $Z_2$, call it $L$.

Still assuming $C_V(T) \neq 0$, suppose in addition $L \leq Z_1$. Then $Y_2$ is centralised by $T$; by Lemma \ref{l:noZiZi+1}, $C_{Y_3}(T) = 0$. But since $w\cdot Z_1$ embeds into $Y_3$ (as a $T$-module and modulo a finite kernel), we find $C_{Z_1}(T) = 0$, that is, $Z_1 = L$.
On the other hand, $\rk C_V(T) = \rk V - 2k > k$. The map $\d : C_V(T) \to V$ has image $\d (C_V(T)) \leq [U, C_V(T)] \leq [T, V]$ by Lemma \ref{l:bracket}, but also $\im \d \leq [U, V] \leq Z_2$, so $\im \d \leq ([T, V] \cap Z_2)\o = L$ of rank exactly $k$, and $\d: C_V(T) \to V$ has infinite kernel: there are infinitely many $c \in C_V(T)$ with $[u, c] = 0$. Applying $T$, any such $c$ is actually centralised by $U$: hence $(C_V(T) \cap Z_1)\o \neq 0$, contradicting $Z_1 = L$, a non-trivial $T$-line.

So, under the assumption $C_V(T) \neq 0$, the only non-trivial $T$-line in $Z_2$ is not contained in $Z_1$: hence $T$ centralises $Z_1$. By Lemma \ref{l:noZiZi+1} again, $C_{Y_2}(T) = 0$, and this shows $Z_2 = Z_1 \oplus L$.
So if an element $v$ lies in $Z_2 \cap w\cdot Z_2$, it has decompositions $v = z + \ell = w z' + \ell'$ with $z, z' \in Z_1$ and $\ell, \ell' \in L$. Now $\ell - w \ell' \in C_V(T) \cap [T, V] = 0$, and $\ell = w \ell' \in L \cap w \cdot L = 0$. Hence $v = z = w z' \in Z_1 \cap w \cdot Z_1$ which is finite by Lemma \ref{l:Bruhat}: we just contradicted $Z_2 \cap w\cdot Z_2$ being infinite.

As a conclusion, $C_V(T) = 0$ and $V$ is therefore the direct sum of non-trivial $T$-lines, the number of which can only be $4$.

But remember that $\check{Z}_2 = (Z_2\cap w \cdot Z_2)\o$ is infinite. Being a proper $\<T, w\>$-submodule of $V$, it is the sum of exactly two non-trivial $T$-lines: $\check{Z}_2 = L \oplus w\cdot L$. Notice that $Z_1 \cap \check{Z}_2 \leq Z_1 \cap w \cdot Z_2$ which is finite by Lemma \ref{l:Bruhat}, hence trivial since $C_V(T) = 0$. This way we reach $V = Z_1 \oplus L \oplus w\cdot L \oplus w \cdot Z_1$.
\end{proofclaim}

\begin{step}\label{p:twist:st:beta}
For $\lambda \in \K^\times$ there is an abelian group isomorphism $\beta_\lambda: Z_1 \simeq L$ such that for any $a \in Z_1$:
\[u_\lambda w a = wa - t_\lambda w \beta_\lambda (a) + \beta_\lambda(a) + t_\lambda a\]
As a matter of fact, $\beta_{\lambda^2}(a) = t_\lambda \beta_1(t_\lambda a)$. Finally, $\beta = \beta_1: Z_1 \to L$ and $\d: L \to Z_1$ are converse bijections.
\end{step}
\begin{proofclaim}
This is very much like Claim \ref{p:quartic:st:beta} in the proof of Proposition \ref{p:quartic}.

Fix $a \in Z_1$. Given the decomposition of Claim \ref{p:twist:st:decomposition}, one should expect additive maps $\alpha: Z_1 \to Z_1$ and $\beta, \gamma: Z_1 \to L$ such that:
\[u w a = w a + w \gamma(a) + \beta(a) + \alpha(a)\]
Bearing in mind that $(u w)^3 = 1$, one finds:
\begin{align}
u w u w \cdot a & = u (a + \gamma(a) + w \beta(a) + w \alpha(a))\nonumber\\
& = a + \gamma(a) + \d \gamma (a) + w \beta(a) + \d w \beta(a)\nonumber\\
& \quad + w \alpha(a) + w \gamma(\alpha(a)) + \beta(\alpha(a)) + \alpha(\alpha(a))\nonumber\\
 = (u w)^{-1}  a & = w u^{-1} a = w  a\label{e:twist}
\end{align}
Projecting onto $w \cdot Z_1$, this yields $\alpha(a) =  a$; projecting onto $w\cdot L$, one finds $\gamma(a) = - \beta(a)$.

Moving from $\beta_1 = \beta$ to $\beta_\lambda$ is exactly as in Claim \ref{p:quartic:st:beta} from the proof of Proposition \ref{p:quartic}

Additivity of $\beta_\lambda$ is obvious; we did not prove that $\beta_\lambda$ is an isomorphism, which will be a consequence of the next.
\end{proofclaim}

\begin{step}\label{p:twist:st:beta:values}
One has $\d_\lambda \beta_\lambda (a) = t_\lambda a$ and $\d_{\lambda^{-1}} w \beta_\lambda (a) = - a$.
For $b \in L$, one has $\d_\lambda wb = - \d_\lambda t_\lambda b$.
\end{step}
\begin{proofclaim}
Equation \eqref{e:twist} from Claim \ref{p:twist:st:beta} now rewrites as $\d w \beta(a) - \d\beta(a) + 2 a = 0$. Let $r = \d\beta(a) \in [U, L] \leq Z_1$, then compute again:
\begin{align*}
uw uw \beta(a) & = uw (w \beta(a) + r - 2a)\\
& = \beta(a) + r + w(r-2a) - w \beta(r - 2a) + \beta(r-2a) + (r-2a)\\
= w^{-1} u^{-1} \beta(a) & = w\beta(a) - w r
\end{align*}
so projecting onto $w \cdot Z_1$ we find $2r = 2a$; the characteristic not being $2$, this results in $a = r = \d\beta(a)$ and $\d w\beta(a) = - a$.
Incidently, this shows that the additive map $\beta: Z_1 \to L$ is injective, and therefore surjective as well. The inverse map is clearly $\d: L \to Z_1$.

Now for $\lambda \in \K$, one has $\d_\lambda \beta_\lambda (a) = t_\ell \d t_\ell^{-1} t_\ell \beta(t_\ell a) = t_\lambda a$ and $\d_{\lambda^{-1}} w \beta_\lambda(a) = t_\ell^{-1} \d t_\ell w t_\ell \beta (t_\ell a) = - a$.

Finally let $b \in L$; since $\beta_{\lambda^{-1}}: Z_1 \to L$ is a bijection, there is $a \in Z_1$ with $b = \beta_{\lambda^{-1}}(a)$. Hence:
\[
\d_\lambda w b = \d_\lambda w \beta_{\lambda^{-1}}(a) = - a = - t_\lambda t_\lambda^{-1} a = - t_\lambda \d_{\lambda^{-1}} \beta_{\lambda^{-1}} a = - \d_\lambda t_\lambda b
\qedhere\]
\end{proofclaim}

\begin{step}\label{p:twist:st:multiplicative}
In $\End(L)$, $\beta \d_\lambda + \beta \d_\mu = \beta \d_{\lambda+\mu}$ and $(\beta \d_\lambda) (\beta \d_\mu) = \beta \d_{\lambda\mu}$ for any $(\lambda, \mu)\in \K^2$.
\end{step}
\begin{proofclaim}
This will be tedious; the author strongly suspects that the second half of the argument could be made shorter with more culture on Jordan rings; incidently, it is not surprising to see Jordan rings appear in a cubic action (see the preprint \cite{GCubic} with the proviso that our setting smoothes many arguments).

Let $\psi_\lambda = \beta \d_\lambda: L \to L$ and $\psi: \K \to \End(L)$ map $\lambda$ to $\psi_\lambda$. In general one has $\d_{\lambda + \mu} = \d_\lambda + \d_\mu + \d_\lambda \d_\mu$; here, since $L \leq Z_2$, $\psi$ is additive. What we must now prove is that $\psi$ is multiplicative as well.

Fix $\lambda$ and $\mu \in \K^\times$. Then in view of Claims \ref{p:twist:st:beta} and \ref{p:twist:st:beta:values}:
\begin{align*}
u_\lambda u_\mu w a & = u_\lambda \left(w a - w t_\mu^{-1} \beta_\mu(a) + \beta_\mu(a) + t_\mu a\right)\\
& = w a - w t_\lambda^{-1} \beta_\lambda(a) + \beta_\lambda (a) + t_\lambda a\\
& \quad - w t_\mu^{-1} \beta_\mu (a) - \d_\lambda w t_\mu^{-1} \beta_\mu(a) + \beta_\mu(a) + \d_\lambda \beta_\mu (a) + t_\mu a\\
= u_{\lambda + \mu} w a & = w a - w t_{\lambda + \mu}^{-1} \beta_{\lambda + \mu}(a) + \beta_{\lambda + \mu}(a) + t_{\lambda + \mu} a
\end{align*}
so projecting on the various summands one finds:
\begin{align}
t_\lambda^{-1} \beta_\lambda (a) + t_\mu^{-1} \beta_\mu(a) & = t_{\lambda +\mu}^{-1}\beta_{\lambda +\mu}(a);\nonumber\\
\beta_\lambda(a) + \beta_\mu(a) & = \beta_{\lambda + \mu}(a);\nonumber\\
t_\lambda a - \d_\lambda w t_\mu^{-1} \beta_\mu(a) + \d_\lambda \beta_\mu(a) + t_\mu a & = t_{\lambda + \mu} a\label{p:twist:eq:prefundamental}
\end{align}
But we also know from Claim \ref{p:twist:st:beta:values} that $t_{\lambda + \mu} a = \d_{\lambda +\mu} \beta_{\lambda + \mu} (a)$; hence using $\im \beta_\lambda = L \leq Z_2$:
\begin{align*}
t_{\lambda + \mu} a & = \d_{\lambda +\mu} \beta_{\lambda + \mu} (a) = \d_\lambda \beta_\lambda (a) + \d_\lambda \beta_\mu(a) + \d_\mu \beta_\lambda(a) + \d_\mu \beta_\mu (a)\\
& = t_\lambda a + \d_\lambda \beta_\mu(a) + \d_\mu \beta_\lambda(a) + t_\mu a
\end{align*}
Putting this with Equation \eqref{p:twist:eq:prefundamental}, one gets $\d_\mu \beta_\lambda(a) = - \d_\lambda w t_\mu^{-1} \beta_\mu(a)$, which equals $\d_\lambda t_\lambda t_\mu^{-1} \beta_\mu(a)$ by Claim \ref{p:twist:st:beta:values}.
Hence:
\begin{equation}
\d_\mu \beta_\lambda(a) = \d_\lambda t_{\lambda\mu^{-1}} \beta_\mu(a)\label{p:twist:eq:fundamental}
\end{equation}

This was unpleasant but the rest of the argument goes slightly better, although we suspect there should be something shorter. We make a series of claims.
\begin{enumerate}
\item\label{p:twist:st:multiplicative:it:beta}
$t_\lambda = \d_\lambda \beta_\lambda$ (as functions on $Z_1$).

This is from Claim \ref{p:twist:st:beta:values}.
\item\label{p:twist:st:multiplicative:it:bd=tpsit}
If $\ell^2 = \lambda$ then $\beta_\lambda \d = t_\ell \psi_\lambda t_\ell$.

Recall from Claim \ref{p:twist:st:beta} that $t_\ell \beta = \beta_\lambda t_\ell^{-1}$. Therefore $t_\ell \psi_\lambda t_\ell = t_\ell \beta \d_\lambda t_\ell = \beta_\lambda t_\ell^{-1} \d_\lambda t_\ell = \beta_\lambda \d$.
\item\label{p:twist:st:multiplicative:it:cool}
If $\ell^2 = \lambda$ and $m^2 = \mu$ then $\psi_\mu t_\ell \psi_\lambda t_\ell = \psi_\lambda t_{\lambda m^{-1}} \psi_\mu t_m$.

This reduces to the previous claim and Equation \eqref{p:twist:eq:fundamental} composed with $\beta$ on the left and $\d$ on the right:
$\psi_\mu t_\ell \psi_\lambda t_\ell = \beta \d_\mu \beta_\lambda \d
 = \beta \d_\lambda t_{\lambda\mu^{-1}} \beta_\mu \d
 = \psi_\lambda t_{\lambda\mu^{-1}} t_m \psi_\mu t_m$.
\item
If $\ell^2 = \lambda$ then $t_\ell \psi_\lambda = \psi_\lambda t_\ell$.

Set $\ell = 1$ in the previous, and find $\psi_\mu = t_m^{-1} \psi_\mu t_m$.
\item
$\psi_\mu t_\lambda \psi_\lambda = \psi_\lambda t_\lambda \psi_\mu$.

With the last two items, $\psi_\mu \psi_\lambda t_\lambda = \psi_\lambda t_\lambda \psi_\mu$; now $\psi_\lambda$ and $t_\lambda$ commute.
\item\label{p:twist:st:multiplicative:it:Hualike}
$\psi_{\mu\lambda^2} = \psi_\lambda \psi_\mu \psi_\lambda$.

This is since $\psi_{\mu\lambda^2} = \beta \d_{\mu\lambda^2} = \beta t_\lambda \d_\mu t_\lambda^{-1}$, so with items \ref{p:twist:st:multiplicative:it:beta} (which applies since $\d_\mu(L) = Z_1$) and \ref{p:twist:st:multiplicative:it:bd=tpsit}, $\psi_{\mu\lambda^2} = \beta \d_\lambda \beta_\lambda \d_\mu t_\lambda^{-1} = \psi_\lambda \beta_\lambda \d \beta \d_\mu t_\lambda^{-1} = \psi_\lambda t_\ell \psi_\lambda t_\ell \psi_\mu t_\lambda^{-1}$; in view of the previous commutation properties, one gets $\psi_{\mu\lambda^2} = \psi_\lambda^2 t_\lambda \psi_\mu t_\lambda^{-1} = \psi_\lambda \psi_\mu t_\lambda \psi_\lambda t_\lambda^{-1} = \psi_\lambda \psi_\mu \psi_\lambda$.

Whether one can conclude at this early stage is not clear to us. Hua's famous theorem \cite{Hua} on (anti-)automorphisms of skewfields would require $\End(L)$, or at least a subring, to be a division ring (which is not obvious even taking $T$-minimality into account); exploiting the identity $\psi_{\lambda^2} = \psi_\lambda^2$ would require it to be commutative.

Perhaps there is an algebraic argument, which we replace by basic model-theory.
%
\item
For any polynomial $P \in \F_p[X]$, $\psi_{P(\lambda)} = P(\psi_\lambda)$.

First show $\psi_{\lambda^n} = \psi_\lambda^n$ by setting $\mu = \lambda^n$ in the previous, then use additivity of $\psi$.
\item
For any two $(\lambda, \mu) \in \overline{\F_p}$, $\psi_\lambda \psi_\mu = \psi_\mu \psi_\lambda$.

This is because there are $\nu \in \overline{\F_p}$ and polynomials $P, Q \in \F_p[X]$ with $\lambda = P(\nu)$ and $\mu = Q(\nu)$.
\item
For any two $(\lambda, \mu) \in \K$, $\psi_\lambda \psi_\mu = \psi_\mu \psi_\lambda$.

For $\lambda \in \overline{\F_p}$,  $C_\lambda = \{\mu \in \K: \psi_\lambda \psi_\mu = \psi_\mu \psi_\lambda\}$ is an infinite, definable field, whence equal to $\K$. So for $\lambda \in \K$, $C_\lambda$ is an infinite, definable field, whence equal to $\K$.
\item
$\psi$ is a ring homomorphism.

This follows from $\psi_\lambda \psi_\mu = \psi_\mu \psi_\lambda$ and $\psi_{\lambda^2} = \psi_\lambda^2$.
\end{enumerate}
This completes the proof of Claim \ref{p:twist:st:multiplicative}.
\end{proofclaim}

Fortunately the rest is now just painting from nature.

\begin{notationinproof}
We define an action of $\K$ on $V$ by letting, for $(\lambda, b) \in \K \times L$:
\[\left\{\begin{array}{rcl}
\lambda \cdot b & = & \beta \d_\lambda b\\
\lambda \cdot wb & = & w (\lambda \cdot b)\\
\lambda \cdot \d b & = & \d (\lambda \cdot b)\\
\lambda \cdot w \d b & = & w (\lambda \cdot \d b)
\end{array}
\right.\]
\end{notationinproof}

\begin{step}
This defines a $\K$-vector space structure on $V$ for which $G$ is linear.
\end{step}
\begin{proofclaim}
First bear in mind that since $\d$ realises a bijection $L \simeq Z_1$, and since $V = Z_1 \oplus L \oplus w \cdot L \oplus w \cdot Z_1$, this is well-defined. Additivity in the module element is obvious. Additivity and multiplicativity in $\lambda$ need only be proved on $L$, and are exactly Claim \ref{p:twist:st:multiplicative}.

It remains to prove that $G$ is linear. By construction, $w$ is. We turn to linearity of $u$, which amounts to that of $\d$. The latter is obvious on $Z_1$ and $L$. Now if $wb \in w\cdot L$ then by Claim \ref{p:twist:st:beta:values}, $\d w (\lambda \cdot b) = - \d (\lambda \cdot b) = - \lambda \cdot \d b = \lambda \cdot \d wb$ and $\d$ is also linear on $w\cdot L$.

Notice that $\beta: Z_1 \to L$ is linear. For if $a \in Z_1$, then there is $b \in L$ with $a = \d b$ or equivalently $b = \beta(a)$, so that:
\[
\beta (\lambda \cdot a) = \beta (\lambda \cdot \d b) = \beta (\d (\lambda \cdot b)) = \lambda \cdot b = \lambda \cdot \beta(a)
\]
Finally, for $a \in Z_1$:
\begin{align*}
\d (\lambda \cdot w a) & = \d w (\lambda \cdot a) = - w \beta (\lambda \cdot a) + \beta (\lambda \cdot a) + (\lambda \cdot a)\\
& = - w \lambda \cdot \beta (a) + \lambda \cdot \beta (a) + \lambda \cdot a\\
& = - \lambda \cdot w \beta (a) + \lambda \cdot \beta (a) + \lambda \cdot a = \lambda \cdot \d w a
\end{align*}

So far we proved linearity of $\<w, u\>$; now to linearity of $T$.
For $b \in L$ one has, thanks to Claim \ref{p:twist:st:multiplicative}:
\begin{align*}
t_\mu (\lambda \cdot b) & = \beta \d t_\mu (\lambda \cdot b) = \beta t_\mu \d_{\mu^{-2}} \beta \d_\lambda b =
\beta t_\mu \d \beta \d_{\mu^{-2}} \beta \d_\lambda b\\
& = \beta t_\mu \d \beta \d_{\lambda\mu^{-2}} b
= \beta t_\mu \d_{\lambda\mu^{-2}} b
= \beta \d_\lambda t_\mu b
= \lambda\cdot (t_\mu b)
\end{align*}
This is linearity of $T$ on $L$; linearity on $w \cdot L$ follows at once. Now for $a \in Z_1$, letting $b = \beta(a)$, one has $\lambda \cdot a = \d (\lambda \cdot b) = \d_\lambda b = \d_\lambda \beta (a)$, so:
\begin{align*}
t_\mu \lambda \cdot a & = t_\mu \d_\lambda \beta (a) = \d \beta \d_{\lambda\mu^2} t_\mu \beta(a)\\
& = \d \beta \d_\lambda \beta \d_{\mu^2} t_\mu \beta (a) = \d_\lambda \beta t_\mu \d \beta (a) = \lambda \cdot (t_\mu a)
\end{align*}
This is linearity of $T$ on $Z_1$; linearity on $w \cdot Z_1$ follows, proving global linearity of $\<w, u, T\> = G$.
\end{proofclaim}

By linearity, $G \simeq \PSL_2(\K)$ embeds definably into $\GL_4(\K)$;
Fact \ref{f:Poizat} we quoted in \S\ref{s:conjecture} guarantees definability of $G$ in the pure field $\K$ together with a finite number of automorphisms. This is certainly consistent with the desired identification but does not settle matters.

Another remarkable information is the use of the Borel-Tits positive answer \cite[Corollaire 10.4]{BTHomomorphismes} to Steinberg's \cite[Conjecture 1.2]{SRepresentations}.
We know at this stage that $V$ either is the rational module of Proposition \ref{p:quartic} (which we already ruled out) or a tensor product $\Nat G \otimes{}^\chi (\Nat G)$ for some field automorphism $\chi$. The problem is that we also announced definability of $\chi$: and for the moment, we see no better argument than retrieving it by a direct computation. \emph{There ought to be something more general here} and we shall return to the topic some day.


\begin{notationinproof}\label{n:chi}
Let $b_0 \in L\setminus\{0\}$ be fixed. For $\lambda \in \K$ let $\chi(\lambda)$ be the unique $\mu \in \K$ with $\d_\mu b_0 + \d_\lambda w b_0 = 0$.
\end{notationinproof}

\begin{step}\label{p:twist:st:chi}
The map $\chi$ is a well-defined, definable automorphism of $\K$ not depending on $b_0$; finally, $V \simeq {}^\chi (\Nat G) \otimes \Nat G$.
\end{step}
\begin{proofclaim}
First notice that the function $U \to Z_1$ mapping $u_\lambda$ to $\d_\lambda b_0$ is a definable bijection. So $\chi$ is a well-defined, definable function on $\K$. Additivity is obvious since $L \oplus w \cdot L \leq Z_2$.

Now remember from Claim \ref{p:twist:st:beta:values} that for $b \in L$, one has $\d_{\lambda} w b = - \d_\lambda t_\lambda b$,
so that:
\[
\chi(\lambda) \cdot b_0 = \beta \d_{\chi(\lambda)} b_0 = - \beta \d_\lambda w b_0 = \beta \d_\lambda t_\lambda b_0 = \lambda \cdot t_\lambda b_0
\]
In particular, by linearity of $T$:
\[
\chi(\lambda\mu) \cdot b_0 = (\lambda\mu)\cdot t_{\lambda\mu} b_0 = \lambda \cdot (\mu \cdot t_\lambda t_\mu b_0) = \lambda \cdot (t_\lambda (\mu \cdot t_\mu b_0)) = \chi(\lambda)\cdot (\chi(\mu) \cdot b_0)
\]
This proves multiplicativity of $\chi$, which is an automorphism of $\K$.

Incidently, if one were to take $b_1 = t \cdot b_0$ for some $t \in T$, one would find:
\begin{align*}
\beta\left(\d_{\chi(\lambda)} b_1 + \d_\lambda w b_1\right) & = \chi(\lambda) \cdot b_1 - \beta \d_\lambda t_\lambda b_1 = \chi(\lambda) \cdot t b_0 - \lambda \cdot t_\lambda t b_0\\
&  = t (\chi(\lambda) \cdot b_0 - \lambda \cdot t_\lambda b_0) = 0
\end{align*}
But $\beta$ is injective, so this means that the function $\chi$ does not depend on the precise choice of $b_0$ in $T \cdot b_0 = L\setminus\{0\}$.

It remains to see that we have a definable isomorphism $V \simeq {}^\chi(\Nat G) \otimes \Nat G$. The latter has basis $(e_1 \otimes e_1, e_1 \otimes e_2, e_2\otimes e_1, e_2 \otimes e_2)$ in the notations of Example \ref{ex:twist}; the reader should take $\varphi = \chi$ and $\psi = \Id$ there.
Let $a_0 = \d b_0 \in Z_1$, so that $b_0 = \beta(a_0)$; clearly $(a_0, b_0, w \cdot b_0, w \cdot a_0)$ forms a basis of $V$ as a $\K$-vector space.
Now consider the map:
\[
f:\left\{
\begin{array}{ccr}
a_0 & \mapsto & e_1 \otimes e_1\\
b_0 & \mapsto & e_1 \otimes e_2\\
w \cdot b_0 & \mapsto & - e_2 \otimes e_1\\
w \cdot a_0 & \mapsto & e_2 \otimes e_2
\end{array}\right.
\]
and extend linearly. We claim that $f$ is $G$-covariant. The action of $w$ and $u$ is obviously preserved. So it suffices to study the action of $T$. Remember that $t_\lambda b_0 = \frac{\chi(\lambda)}{\lambda} \cdot b_0$. Hence $t_\lambda a_0 = t_\lambda \d b_0 = \d_{\lambda^2} t_\lambda b_0 = \frac{\chi(\lambda)}{\lambda} \d \beta \d_{\lambda^2} b_0 = \lambda \chi(\lambda) a_0$. Therefore:
\[f(t_\lambda a_0) = \lambda \chi(\lambda) f(a_0) = \lambda \chi(\lambda) e_1\otimes e_1 = t_\lambda f(a_0)\]
whereas:
\[f(t_\lambda b_0) = \frac{\chi(\lambda)}{\lambda} f(b_0) = \frac{\chi(\lambda)}{\lambda} e_1\otimes e_2 = t_\lambda f(b_0)\]
Covariance on $w a_0$ and $w b_0$ follows.
\end{proofclaim}

This concludes the proof of Proposition \ref{p:twist} and the rank $4k$ analysis.
\end{proof}

\begin{remarks*}\
\begin{itemize}
\item
In order to convince the reader that despite our efforts matters are not perfectly clear, we do not know if one can classify cubic $\SL_2(\K)$-modules of finite Morley rank, i.e. those satisfying $\ell_U(V) = 3$.
\item
Several years ago the author had tried something without model theory: there are many obstacles.

First, the decomposition of Claim \ref{p:twist:st:decomposition} cannot be granted.
Second, it is not clear that injectivity of $\beta$ entails its surjectivity in Claim \ref{p:twist:st:beta:values}.
Also, but this may be due to our lack of algebraic culture, multiplicativity of $\psi$ in Claim \ref{p:twist:st:multiplicative} apparently holds only over the algebraic closure of the prime field (to be more precise, the argument works over any quadratically closed, algebraic extension of $\F_p$).
Last but not least, in order to ``untensor'' definably, one needs to construct $\chi$ as in Notation \ref{n:chi}, which requires $\lambda \mapsto \d_\lambda b_0$ to be a bijection $\K \simeq \Z_1$; the latter is not clear without rank arguments.

This is the reason why we finally prefered to work in finite Morley rank. One's first love is not easily forgotten.
\end{itemize}
\end{remarks*}

\bibliography{../../Timmesfeld-Variationen/English/Variationen,RMf}
\bibliographystyle{plain}

\end{document}